\documentclass[a4paper,12pt]{amsart}

\RequirePackage{kpfonts}

\usepackage[nameinlink,capitalise]{cleveref}
\usepackage{xurl,amssymb,color,todonotes}

\newcommand{\norm}[1]{\left\Vert#1\right\Vert}
\renewcommand{\sp}[2]{\left<#1, #2\right>}
\newcommand{\Psiiid}{\overline\Psi_{N}}
\newcommand{\prob}[1]{\P\left(#1\right)}

\newcommand{\inv}{^{-1}}

\newcommand{\R}{\mathbb{R}}
\newcommand{\E}{\mathbb{E}}
\newcommand{\N}{\mathbb{N}}
\newcommand{\Z}{\mathbb{Z}}

\renewcommand{\P}{\mathbb P}

\newcommand{\BX}{\overline{X}}
\newcommand{\wphi}{\widehat{\varphi}}

\DeclareMathOperator*{\argmin}{argmin}
\newcommand{\abs}[1]{\left| #1 \right|}

\newcommand{\lan}{\langle}
\newcommand{\ran}{\rangle}

\renewcommand{\d}{\textnormal{d}}
\newcommand{\opr}{{\operatorname{op}}}
\newcommand{\vphi}{\varphi}
\newcommand{\vp}{\varphi}
\newcommand{\de}{\delta}
\newcommand{\De}{\Delta}

\newcommand{\ga}{\gamma}
\newcommand{\al}{\alpha}

\theoremstyle{plain}
\newtheorem{prop}{Proposition}[section]
\newtheorem{cor}[prop]{Corollary}
\newtheorem{theo}[prop]{Theorem}
\newtheorem{lem}[prop]{Lemma}

\newcommand{\bee}{\begin{equation}}
\newcommand{\eee}{\end{equation}}

\theoremstyle{definition}

\newtheorem{rem}[prop]{Remark}
\theoremstyle{remark}
\newtheorem{assumption}{Assumption}

\begin{document}
\title[Estimation of particle systems]{On nonparametric estimation of the interaction function in particle system models}
\author[D.~Belomestny]{Denis Belomestny$^{1}$}
\address{$^1$Faculty of Mathematics\\
Duisburg-Essen University\\
Thea-Leymann-Str.~9\\
D-45127 Essen\\
Germany}
\email{denis.belomestny@uni-due.de}

\author[M.~Podolskij]{Mark Podolskij$^{2}$}
\address{$^2$ Department
of Mathematics, University of Luxembourg \\
Luxembourg }
\email{mark.podolskij@uni.lu}

\author[S.-Y.~Zhou]{Shi-Yuan Zhou$^{3}$}
\address{$^3$ Department
of Mathematics, University of Luxembourg \\
Luxembourg }
\email{shi-yuan.zhou@uni.lu}

\date{}
\maketitle

\begin{abstract}
\noindent This paper delves into a nonparametric estimation approach for the interaction function within diffusion-type particle system models. We introduce two estimation methods based upon an empirical risk minimization. Our study encompasses an analysis of the stochastic and approximation errors associated with both procedures, along with an examination of certain minimax lower bounds. In particular, we show that there is a natural metric under which the corresponding minimax estimation error of the interaction function  converges to zero with parametric rate.  This result is rather suprising given  complexity of the underlying estimation problem and rather large classes of interaction functions for which the above parametric rate holds.
 \\ \\


\end{abstract}

\section{Introduction} \label{sec1}
\noindent This article is dedicated to the study of $N$-dimensional interacting particle systems described by the equation 
\begin{equation}\label{model}
dX_t^{i,N} = (\vphi \star \mu_t^N) (X_t^{i,N})\, dt + \sigma dW_t^{i}, \quad i=1,\ldots, N, 
\end{equation}
where \(t\in [0,T], \) $\vphi\colon\R\to\R$ is the interaction potential and $(W^i)_{i=1}^N$ are independent one-dimensional Brownian motions.
Here, $\mu_t^N$ stands for the empirical measure of the particle system at time $t$, given by 
\begin{equation}
 \mu_t^N := \frac 1N \sum_{i=1}^N \de_{X_t^{i,N}},
\end{equation}
and $\vphi \star \mu (x) := \int_{\R} \vphi(x-y) \mu(dy)$. 
We make the assumption that we observe 
\(N\)  paths $(X_t^{i,N}, t\in [0,T],i=1,\ldots,N) $ and aim to estimate the unknown interaction function $\vphi$ as $N\to \infty$ with $T>0$ being fixed. It is noteworthy that, given complete path observations, the diffusion coefficient $\sigma$ can be considered known.
\par
Interacting particle systems of the type \cref{model} play an important role in probability theory and applications.  Initially introduced by McKean in pioneering work \cite{M66}, these systems served as models for plasma dynamics. In more recent years, diffusion-type interacting particle systems have found wide-ranging applications in finance \cite{CDLL19, DGZZ, FS13, GSS19, GH11}, social science \cite{CJLW17}, neuroscience \cite{BFFT12}, and population dynamics \cite{ME99}, among others. From a probabilistic perspective, the particle system delineated in \cref{model} is inherently linked to its corresponding mean-field limiting equation—a one-dimensional McKean-Vlasov stochastic differential equation (SDE),
\begin{equation} \label{eq:McKeanEq}
\left\{
\begin{array}{ll}
	dX_t =  (\vphi \star \mu_t )(X_t)\, dt + \sigma dB_t,\\ \textnormal{Law}(X_t)=\mu_t,
\end{array}
\right.
\end{equation}
where $B$ is a $1$-dimensional Brownian motion. Under suitable assumptions on the function $\vphi$, a well-known phenomenon known as the \textit{propagation of chaos} emerges. This phenomenon indicates that, as $N \rightarrow \infty$, the empirical probability measure $\mu_t^{N}$ weakly converges to $\mu_t$ for any $t>0$, and the McKean-Vlasov stochastic differential equation (SDE) \cref{eq:McKeanEq} characterizes the asymptotic trajectory of an individual particle. Classical treatments of the propagation of chaos and McKean-Vlasov SDEs can be found in the monographs \cite{K10,S91}. Furthermore, the measure
$\mu_t$ is absolutely continuous with respect to the Lebesgue measure for all $t>0$, assuming certain regularity conditions on $\vphi$.
Moreover, the density of $\mu_t$, denoted again by $\mu_t$
(with a slight abuse of notation), satisfies the granular media equation:
\begin{eqnarray}
\label{eq:pde}
\partial_t\mu_t=\partial_x\{(\sigma^2/2) \partial_x\mu_t+(\varphi\star\mu_t)\mu_t\}.
\end{eqnarray}
The realm of statistical inference for interacting particle systems remains a less explored domain. While a considerable number of recent articles delve into parametric estimation for particle systems and McKean-Vlasov stochastic differential equations under diverse model and sampling assumptions (see, for instance, \cite{AHPP22, B11, C21, CG23, DM22b, GL21a, GL21b, K90, L22, SKPP21, WWMX16}), non- and semiparametric estimation of the interaction potential $\vphi$ presents a notably more intricate challenge. This complexity arises from the convolution with measures $\mu_t^N$ or $\mu_t$ and has only been the subject of investigation in three recent papers. Specifically, the articles \cite{ABPPZ23,BPP22} explore the semiparametric estimation of $\vphi$ given observations $X_T^{i,N},i=1,\ldots, N,$ with $N,T\to \infty$. The methodology in \cite{ABPPZ23,BPP22} heavily relies on the existence of the invariant density, necessitating the asymptotic regime $T \to \infty,$ and employs the deconvolution method.
Another relevant work, \cite{DM22a}, focuses on the same observation scheme as the current paper, presenting a nonparametric approach for the drift function $\vphi \star \mu_t$ but demonstrating consistency only for an estimator of the interaction function $\vphi$.
\par
This paper aims to investigate nonparametric estimation of the function $\vphi$ using the sampling scheme $X_t^{i,N}, t\in [0,T], i=1,\ldots,N,$ with $N\to \infty$ and $T$ held fixed. We propose two interconnected methods, both based on empirical risk minimization over sieves $(S_m)_{m\geq 1}$, employing an unconventional underlying loss function $\|\cdot\|_{\star}$ that depends on the unknown measure $(\mu_t)_{t\in [0,T]}$.
Our first approach addresses the scenario where the functional spaces $(S_m)_{m\geq 1}$ are compact 
 with respect to $\|\cdot\|_{\infty}$. In this case, the analysis of the proposed estimator's rate heavily relies on empirical processes theory for (degenerate) U-statistics.
 The second approach investigates a more traditional setting, where
$(S_m)_{m\geq 1}$ are finite-dimensional vector spaces. Here, we require a further truncation of the empirical risk minimizer, akin to the approach introduced in \cite{CG20b} for classical stochastic differential equations, where the coefficients are independent of the laws $(\mu_t)_{t\in [0,T]}$. We derive convergence rates with respect to the loss function $\|\cdot\|_{\star}$, involving an examination of stochastic and approximation errors. We will  demonstrate that the resulting convergence rate is often $\sqrt{\log(N)/N}$
in the first scenario, which essentially matches the minimax lower bound. However, for the second method, the convergence rate decreases due to the truncation procedure.
Additionally,  
we will establish that the minimax lower bound for the traditional $L^2$-loss is logarithmic, consistent with the findings of \cite{BPP22}.
\par
The paper is structured as follows. Section \ref{sec2} concentrates on a risk minimization approach over  compact functional classes. This section elucidates the core concepts, establishes convergence rates in terms of metric entropy, and analyses the approximation error associated with the estimation method. In Section \ref{sec2ab}, we introduce an alternative approach that involves risk minimization over vector spaces coupled with an additional dimension truncation. Uniform convergence rates for the resulting estimator are presented. Section \ref{sec.low} delves into an exploration of minimax lower bounds. Specifically, we demonstrate that the estimator developed in Section \ref{sec2} attains essentially optimal rates with respect to the $\|\cdot\|_{\star}$-norm. Additionally, we establish a logarithmic lower bound when considering the traditional $L^2$-loss. For the detailed proofs of the main results, we refer to Section \ref{sec5.proofs}.

\subsubsection*{Notation}
All random variables and stochastic processes are defined on a filtered probability space $(\Omega,\mathcal{F},
(\mathcal{F}_t)_{t\geq 0}, \mathbb{P})$. Throughout, positive constants are represented by $C$ (or $C_p$ if dependent on an external parameter $p$), though their values may vary across lines. We use the notation
$C^k(\R)$ (resp. $C_c^2(\R)$) to denote the space of $k$ times continuously differentiable functions (resp. the space of $k$ times continuously differentiable functions with compact support). 
The sup-norm of a function $f:\mathbb{R}\to \mathbb{R}$ is denoted as $\|f\|_{\infty}$ or $\|f\|_Q:=\sup_{x\in Q} |f(x)|$. When referring to a probability measure or density $\mu$, $\widehat{\mu}$ denotes its Fourier transform:
$$\widehat \mu (z) = \int_\R \exp(izx) \mu(x) \, dx.$$ 
Given a matrix $A\in \R^{k\times k}$, we denote by $\|A\|_{\text{op}}$ the operator norm of $A$; we write $A\succcurlyeq B$ if $A-B$ is a positive semidefinite matrix.
For two sequences $(a_n)_{n\in \N}$ and $(b_n)_{n\in \N}$
we write $a_n \lesssim b_n$ if there exists a constant $C>0$ with $a_n\leq Cb_n$ for all $n\in \N$.
The Orlicz norm of a real valued random variable $\eta$ with respect to a nondecreasing, convex
function $\psi$ on $\mathbb{R}$ with \(\psi(0)=0\) is defined by 
\[
\Vert
\eta\Vert_{\psi}:=\inf\left\{t>0:\mathbb{E} \psi\bigl(|\eta|/t\bigr)  \leq 2\right\}.
\]
When \(\psi_p(x):=x^p\) with \(p\geq 1\) the corresponding Orlicz norm is (up to a constant) the \(L^p\)-norm 
\(\|\eta\|_p=(\E[|\eta|^p])^{1/p}.\)
We say that $\eta$ is
\emph{sub-Gaussian} if $\Vert\eta\Vert_{\psi_{e,2}}<\infty$ with \(\psi_{e,2}:=\exp(x^2)-1.\) In particular,
this implies that for some constants $C,c>0$,
\[
\mathbb{P}(|\eta|\geq t)\leq2\exp\left(  -\frac{ct^{2}}{\Vert\eta
\Vert_{\psi_{e,2}}^{2}}\right)  \text{\ \ and \ \ }\mathbb{E}[|\eta|^{p}%
]^{1/p}\leq C\sqrt{p}\Vert\eta\Vert_{\psi_{e,2}},\quad p\geq1.
\]
Consider a real valued random process $(X_{t})_{t\in\mathcal{T}}$ on a metric
 space $(\mathcal{T},\mathsf{d})$. We say that the process has
\emph{sub-Gaussian increments} if there exists $C>0$ such that
\[
\Vert X_{t}-X_{s}\Vert_{\psi_{e,2}}\leq C\mathsf{d}(t,s),\quad\forall
t,s\in\mathcal{T}.
\]
Let $(\mathsf{Y},\rho)$ be a metric space and $\mathsf{X\subseteq Y}$. For
$\varepsilon>0$, we denote by $\mathcal{N}(\mathsf{X},\rho,\varepsilon)$ the
covering number of the set $\mathsf{X}$ with respect to the metric $\rho$,
that is, the smallest cardinality of a set (or net) of $\varepsilon$-balls in
the metric $\rho$ that covers $\mathsf{X}$. Then $\log\mathcal{N}%
(\mathsf{X},\rho,\varepsilon)$ is called the {\it metric entropy}  of $\mathsf{X}$ and
\[
\mathsf{DI}(\mathsf{X},\rho,\psi):=\int_{0}^{\operatorname{diam}(\mathsf{X})}\psi^{-1}\bigl(\mathcal{N}(\mathsf{X},\rho,u)\bigr)\,du
\]
with $\operatorname{diam}(\mathsf{X}):=\max
_{x,x^{\prime}\in\mathsf{X}}\rho(x,x^{\prime}),$ is called the Dudley
integral  with respect to \(\psi.\) For example, if $|\mathsf{X}|<\infty$ and $\rho(x,x^{\prime
})=1_{\{x\neq x^{\prime}\}}$ we get $\mathsf{DI}(\mathsf{X},\rho,\psi_{e,2})=\sqrt
{\log|\mathsf{X}|}.$

\section{Risk minimization over compact functional classes} \label{sec2}

\noindent We start by introducing assumptions on the underlying interacting particle system that insure the propagation of chaos phenomenon. We assume that the initial law of the model is given by $\textnormal{Law}(X_0^{1,N},\ldots,X_{0}^{N,N})=\mu_0^{\otimes N}$, where $\mu_0^{\otimes N}$ denotes the $N$-fold product measure of some probability measure $\mu_0$ with
$\int x^k \mu_0(dx)<\infty$ for all $k \in \N$.
Furthermore, we assume the following condition.

\begin{assumption}
	\label{assLip}
	 The interaction function $\vphi:\R \to \R$ is globally Lipschitz continuous  and bounded, that is, 
\begin{eqnarray*}
|\vphi(x)-\vphi(y)|\leq L_\vphi|x-y|,\quad |\vphi(x)|\leq K_\vphi,\quad x,y\in \R
\end{eqnarray*}
for some finite $L_\vphi,K_\vphi>0.$
\end{assumption}

\noindent This condition guarantees the validity of propagation of chaos in the sense that, for all $t\in[0,T]$, $\mu_t^N$ converges weakly to $\mu_t$ (see e.g. \cite[Theorem 3.1]{CD22}).

\subsection{Construction of the estimator and main results}
We initiate our exposition by outlining the fundamental principles of our estimation methodology. To begin, we consider a sequence of spaces $(S_m)_{m\geq 1}$ (sieves). The crucial idea of our approach lies in the following minimization strategy:
\begin{equation}
\min_{f\in S_N} \frac{1}{NT}\sum_{i=1}^N \int_0^T \left( f \star \mu_t^N (X_t^{i,N}) - \vphi \star \mu_t^N (X_t^{i,N})   \right)^2 dt.
\end{equation}
However, the above risk function cannot be directly computed from the data since the interaction function $\vphi$ is unknown. We derive an empirical (noisy) version of the minimization problem by omitting the irrelevant term $ \vphi \star \mu_t^N (X_t^{i,N})^2$ in the integrand and minimizing the resulting quantity:
\begin{align*} 
\ga_N(f):= 
 \frac{1}{NT}\sum_{i=1}^N\left(  \int_0^T f \star \mu_t^N (X_t^{i,N})^2 \, dt - 2\int_0^T f \star \mu_t^N (X_t^{i,N})\, dX_t^{i,N}    \right) \nonumber
\end{align*}
over $S_N$. For further analysis, we introduce the following bilinear forms:
\begin{align*}
\lan f,g \ran_N &:=  \frac{1}{NT} \sum_{i=1}^{N} \int_0^T (f \star  \mu_t^N) (X_t^{i,N})
 (g \star  \mu_t^N) (X_t^{i,N})\, dt, \\[1.5 ex]
 \lan f,g \ran_{\star} &:= \frac{1}{T} \int_0^T \int_{\R} (f \star  \mu_t) (x)
 (g \star  \mu_t )(x) \mu_t(x)\, dx\, dt.
\end{align*}
We set  $\|f\|_N^2:= \lan f,f \ran_N$ and $\|f\|_{\star}^2:= \lan f,f \ran_{\star}$. With these notations at hand, we deduce the identity
\bee
\label{eq:ident-norm}
\|f-\vp\|_{N}^2 = \ga_N(f) + \ga'_N(f)+ \|\vp\|_{N}^2 
\eee
where $\ga'_N(f)$ is the martingale term defined via
\[
\ga'_N(f):= \frac{2\sigma}{NT} \sum_{i=1}^N \int_0^T (f \star \mu_t^N)(X_t^{i,N})\, dW_t^i.
\]
Finally, our estimator $\vp_N$ is defined as follows:
\bee \label{defphiN}
\vp_N:= \text{argmin}_{f\in S_N} \ga_N(f), \qquad \vp_{\star}:= \text{argmin}_{f\in S_N} \|f-\vp\|_{\star}^2.
\eee
The identity \eqref{eq:ident-norm} means that by minimizing \(\ga_N(f)\) we minimize (up to a martingale term $\ga'_N(f)$) \(\|f-\vp\|_{N}\) which is close to \(\|f-\vp\|_{\star}\) for \(N\) large enough.
\par
Before we proceed with the asymptotic analysis, we give some remarks about a rather unconventional loss function $\|\cdot\|_{\star}$ and its relation to the classical $L^2(\R)$-norm.

\begin{rem} \label{rem1}
The upcoming analysis will reveal that the risk $\|\cdot\|_{\star}$ is the intrinsic norm for evaluating error bounds. Unlike conventional literature on empirical risk minimization, it is crucial to note that this norm is elusive since the laws $\mu_t$ remain unobserved. This presents one of the principal mathematical challenges in our statistical analysis. \qed 
\end{rem}

\begin{rem} \label{rem2}
The norm $\| \cdot \|_{\star}$ can be bounded from above by the $L^2(\R)$-norm. Indeed, 
if we assume that the densities $(\mu_t)_{t\in[0,T]}$ satisfy $\int_0^{T}\|\mu_t\|_{\infty}\, dt<{\infty}$, 
we conclude by the Young's convolution inequality that 
\[
\|f\|_{\star}^2 \leq \frac 1T \int_0^T \|\mu_t\|_{\infty} \|f\star \mu_t\|^2_{L^2(\R)}\, dt
\leq \|f\|^2_{L^2(\R)}\frac{1}{T} \int_0^{T}\|\mu_t\|_{\infty} \,dt. 
\] 
Nonetheless, these two norms are far from being equivalent. To illustrate this point, we assume, for simplicity, that $\mu_t=\mu$ represents the density of the standard normal distribution for all $t$. Now, consider the function $f(x)=1_{[2M,3M]}(x)$ for some $M>0$. In this scenario, it is evident that $\| f \|_{L^2(\R)}^2=M$. On the other hand, there exist constants $c_1, c_2>0$ such that
\begin{align*}
\|f\|_{\star}^2 &= \int_{[-M,M]} (f \star  \mu)^2 (x)
  \mu(x)\,dx + \int_{[-M,M]^c} (f \star  \mu)^2 (x)
  \mu(x)\, dx \\
  &\leq c_1 \exp(-c_2 M^2).
\end{align*}
In particular,  $\| f \|_{L^2(\R)} \to \infty$  while $\|f\|_{\star} \to 0$ as $M\to \infty$ at an exponential rate. 
\qed 
\end{rem} 
To formulate our main statements, we assume that the spaces  $(S_m)_{m\geq 1}$
 are compact with respect to $\|\cdot\|_{\infty}$ 
 with $\|f\|_{\infty}\leq K_{\vphi}$ for $f\in S_m$.
 In certain cases, we employ the notation $S_m(K_{\vphi})$ to explicitly express the dependency of spaces on the parameter $K_{\vphi}$. We assume that the following property holds:
 \[
 f,g \in S_m(K_{\vphi})  \implies f\pm g \in S_m(2K_{\vphi}).
 \]
Subsequently, we establish the bounds for the error term \(\|\vp_N-\vp\|_{\star}.\)
\begin{theo} \label{Th2.3}
It holds under Assumption~\ref{assLip},  
\begin{multline*}
\Bigl\{\E\bigl[\|\vp_N-\vp\|^q_{\star}\bigr]\Bigr\}^{1/q}\leq  \|\vp_{\star}-\vp\|_{\star}+C \left(\frac{K_\vphi\,\mathsf{DI}(S_N(2K_{\vphi}),\|\cdot\|_\infty,\psi_{e,2})}{\sqrt{N}} \right.
\\
\left.+\frac{p^2 K_\vphi\,  \mathsf{DI}(S_N(2K_{\vphi}),\|\cdot\|_\infty,\psi_{p/2})}{N}+\frac{p^3 K_\vphi\,  \mathsf{DI}(S_N(2K_{\vphi}),\|\cdot\|_\infty,\psi_{p/3})}{N^{3/2}} \right)
\end{multline*}
for any $p>2,$ \(2\leq q\leq p,\) where $C>0$ is an absolute  constant  and $\vphi^\star$ stands for the best approximation of $\vphi$ as defined in \eqref{defphiN}. 
\end{theo}
\noindent 
Let us outline the key concepts behind the proof of Theorem \ref{Th2.3}. We commence with the inequality:
\begin{equation}
\label{eq:ineq-oracle}
\|\vp_N -\vp\|_{\star}^2 \leq \|\vp_{\star} -\vp\|_{\star}^2 
+ 2\left(\sup_{f\in S_N} \left| \|f-\vp\|_{N}^2 - \|f -\vp\|_{\star}^2  \right| +  \sup_{f\in S_N} 
|\ga'_N(f)| \right).
\end{equation}
This inequality decomposes the estimation error into an approximation error \(\|\vp_{\star} -\vp\|_{\star}^2\) and stochastic errors related to the last two terms. To handle the stochastic errors, we initially consider i.i.d.~observations from the McKean-Vlasov SDE \eqref{eq:McKeanEq}. Remarkably, the term \(\|f-\vp\|_{N}^2 - \|f -\vp\|_{\star}^2\) turns out to be a linear combination of (degenerate) U-statistics, while \(\ga'_N(f)\) represents a martingale, as noted earlier. Concentration inequalities are applied to both terms, leading to the moment bound in Theorem \ref{Th2.3} via the metric entropy. In the final step of the proof, we employ a change of measure device initially proposed in \cite{DM22a} to transfer the statement from i.i.d.~observations of the McKean-Vlasov SDE \eqref{eq:McKeanEq} to observations of the original particle system \eqref{model}.
In case we have a bound on the covering number of the space $S_N$, we immediately obtain the following corollary. 
\begin{cor}
\label{cor: vc-class}
Suppose that \(\mathcal{N}\bigl(S_N(2K_{\vphi}),\|\cdot\|_\infty,\varepsilon\bigr) \lesssim \varepsilon^{-D_N}\) for all \(N\) and an increasing sequence of positive numbers \(D_N\) satisfying \(D^2_N N^{-1/2}\to 0\) as \(N\to \infty.\) Then it holds
\begin{eqnarray*}
\Bigl\{\E\bigl[\|\vp_N-\vp\|^q_{\star}\bigr]\Bigr\}^{1/q}\leq  \|\vp_{\star}-\vp\|_{\star}+C K_\vphi\,\sqrt{\frac{D_N}{N}}.
\end{eqnarray*}
\end{cor}

\noindent
Our focus now shifts to examining the approximation error $\|\vp_{\star}-\vp\|_{\star}$, which will be addressed in the subsequent analysis.

\subsection{Approximation error}

\noindent In this segment, we delve into an examination of the approximation error outlined in Theorem \ref{Th2.3}. The strong contractivity inherent in the norm $\|\cdot\|_{\star}$ introduces a surprising characteristic wherein the error $\|\vp_{\star}-\vp\|_{\star}$ often exhibits exponential decay in the dimensionality of the space $S_N$.
In the following, we adopt the subsequent assumption regarding the initial distribution:

\begin{assumption}
	\label{assmu0}
Assume that the initial distribution is normal, that is,	\[
\mu_0(x)=\frac{1}{\sqrt{2\pi \zeta^2}} \exp(-x^2/(2\zeta^2))
\]
for some \(\zeta>0.\) 
\end{assumption}

\noindent We also consider a condition on the interaction function $\varphi$:
\begin{assumption}
	\label{assphi}
There is a monotone increasing sequence \((\alpha_n)\) with
\[
\alpha_n\to \infty,\quad  \alpha_{n+1}/\alpha_n\to 1, \quad  n\to \infty
\]
such that the following representation holds
\begin{eqnarray*}
(\mu_t\star\varphi)(x)  =\frac{1}{\sqrt{\alpha_n}}\sum_{k=0}^{\infty}c_{k}\exp(i\pi kx/\alpha_n)\widehat{\mu}_t(-k/\alpha_n), \quad x\in [-\alpha_n,\alpha_n] 
\end{eqnarray*}
with coefficients $c_k=c_k(\alpha_n)$ satisfying $\sum_{k=1}^{\infty}|c_k|\lesssim K_\vphi \sqrt{\alpha_n}.$ 
\end{assumption}

\begin{rem}
    Let us examine Assumption \ref{assphi}.  This assumption holds for any function $\varphi$, which  admits a Fourier series representation of the form:
    \begin{align} \label{Fourie}
    \varphi(x)  =\frac{1}{\sqrt{\alpha}}\sum_{k=0}^{\infty}a_{k}\exp(i\pi kx/\alpha),\quad x\in \mathbb{R},
    \end{align}
    with absolutely summable coefficients $(a_k)_{k\geq 0}$ satisfying $\sum_k |a_k|\leq \sqrt{\alpha}K_\vphi$ and some $\alpha>0.$ This type of periodic potentials appear in McKean-Vlasov equations arising in physics, see e.g. the Kuramoto-Shinomoto-Sakaguchi model in \cite{Frank}. Indeed, for such functions we readily deduce the identity
    \[
    (\mu_t\star\varphi)(x) = \frac{1}{\sqrt{\alpha}}\sum_{k=0}^{\infty}a_{k}
    \exp(i\pi kx/\alpha)\widehat{\mu}_t(-k/\alpha).
    \]
    Taking, for instance, $\alpha_n=\alpha n$ and comparing the coefficients,
    we obtain the statement of Assumption \ref{assphi} with $c_{kn}=a_k \sqrt{n}$ and $c_l=0$ for $l\not \in \{kn:~k\geq 0\}$.
   \qed
\end{rem}

\noindent
Next,
for any $A>1,$ $D\in \mathbb{N},$ we consider the functional space
\begin{equation*}
    S(A,D):=\left\{x\to  \frac{1}{\sqrt{A}}\sum_{k=0}^D c_k\exp(i\pi kx/A): \quad \sum_{k=1}^D|c_k| \leq \sqrt{A}K_\vphi\right\}.
\end{equation*}
The main result of this subsection is the following theorem.

\begin{theo} \label{approxerr}
Set with arbitrary small $\delta>0,$
\[
A_N:=\sqrt{2(\zeta^2+T)\log(N)}, \qquad D_N:=\sqrt{(2+\delta)\frac{\zeta^2+T}{\zeta^2}}\log(N)
\] and define $n_{N}:=\min\{n: \alpha_n\geq A_N\},$ $N\in \mathbb{N}.$
Then under Assumptions \ref{assmu0} and \ref{assphi}, we obtain 
\begin{eqnarray*}
\inf_{g\in S_N=S(\alpha_{n_N},D_N)} \|g-\vp\|_{\star}^2\lesssim \frac{\exp(K_\vphi^2 T)}{N}.
\end{eqnarray*}
As a consequence, we deduce that 
\begin{eqnarray*}
\Bigl\{\E\bigl[\|\vp_N-\vp\|^2_{\star}\bigr]\Bigr\}^{1/2}\lesssim  \sqrt{\frac{\exp(K_\vphi^2 T)}{N}}+K_\vphi\,\sqrt{\frac{\log(N)}{N}}.
\end{eqnarray*}
Here $\lesssim$ stands for inequality up to an absolute constant. 
\end{theo}

\noindent
As a consequence of Theorem \ref{approxerr}, the resulting convergence rate is established as $\sqrt{\log(N)/N}$. We will demonstrate in Section \ref{sec.low} that this rate is essentially optimal. However, it is important to note, as highlighted in Remark \ref{rem2}, that this does not necessarily imply a polynomial convergence rate with respect to the classical $\|\cdot\|_{L^2(\R)}$-norm. In fact, we will establish in Section \ref{sec.low} that the minimax lower bound for the $\|\cdot\|_{L^2(\R)}$-norm is logarithmic, consistent with the findings of \cite{BPP22}.

\begin{rem} Theorem \ref{approxerr} suggests that the approximation error decays exponentially when dealing with functions of the form described in \cref{Fourie}. However, this phenomenon is not exclusive to the Fourier basis. Remarkably, a similar result can be observed for the polynomial basis. The forthcoming explanation sheds light on the fast rates of approximation in the $\|\cdot \|_{\star}$-norm.
Suppose that $\varphi\in L^{1}(\mathbb{R})$. Then we have by the Parseval identity 
\[
(\varphi\star\mu_{t})(x+iy)=\frac{1}{2\pi}\int_{\mathbb{R}} \exp(iux-uy)\widehat{\varphi}(u)\widehat{\mu}_{t}(-u)\,du.
\]
Under Assumption \ref{assmu0}, $|\widehat{\mu}_{t}(u)|\leq\exp(-u^{2}\zeta^{2}/2)$ (see \cref{muhatbounded})
and hence
\begin{align*}
|(\varphi\star\mu_{t})(x+iy)| & \leq\frac{\|\varphi\|_{L^{1}}}{2\pi}\int \exp(-uy)|\widehat{\mu}_{t}(u)|\,du\\
 & \leq\sqrt{\frac{1}{2\pi\zeta^{2}}}\|\varphi\|_{L^{1}}\exp\left(y^{2}/(2\zeta^{2})\right).
\end{align*}
Thus 
\[
\max_{|z|=r}|(\varphi\star\mu_{t})(z)|\leq\sqrt{\frac{1}{2\pi\zeta^{2}}}\|\varphi\|_{L^{1}}\exp\left(r^{2}/(2\zeta^{2})\right), \quad r>0.
\]
As a result, the function $\varphi\star\mu_{t}$ is entire of order $2$ and finite type. Furthermore,  we deduce the power series expansion
\[
(\varphi\star\mu_{t})(z)=\sum_{n=0}^{\infty}c_{n}z^{n}
\]
with $c_{n}\leq\left(e/\zeta^{2}n\right)^{n/2}.$ This implies the inequality
\begin{align*}
\sup_{x\in[-A,A]}\left|(\varphi\star\mu_{t})(x)-\sum_{n=0}^{D}c_{n}x^{n}\right| & \leq\varepsilon_{D,A}=\sum_{n=D+1}^{\infty}\left(\frac{eA^{2}}{\zeta^{2}n}\right)^{n/2}
\end{align*}
Note that 
\[
\varepsilon_{D,A}\leq2e^{-(1+D)/2}
\]
provided $(1+D)^{-1}A^{2}\leq\zeta^{2}/e^{2}.$ Let $T_{K}[f]=1_{\{|f(x)|\leq K\}}f(x)+K\mathrm{sign}(f(x))$
be a truncation operator at level $K$. Then it holds that 
\[
\inf_{\psi\in\mathrm{span}\{z^{n},n\leq D\}}\frac{1}{T}\int_{0}^{T}\int\left((\mu_{t}\star\varphi)(x)-T_{K_{\varphi}}[\psi\star\mu_{t}](x)\right)^{2}\mu_{t}(x)\,dx\,dt\leq\varepsilon_{D,A}^{2}+R_{T}(A)
\]
 with 
\[
R_{T}(A)=\frac{4K_{\varphi}^{2}}{T}\int_{0}^{T}\int_{|x|>A}\mu_{t}(x)\,dx\,dt.
\]
Here we used the fact that for any $\psi\in\mathrm{span}\{z^{n}:~n\leq D\}$ it holds $\psi\star\mu_{t}\in\mathrm{span}\{z^{n}:~n\leq D\}.$
Similarly to the case of Fourier basis (see proof of Theorem~\ref{approxerr}), we get 
\[
R_{T}(A)\le8K_{\varphi}^{2}\exp\left(K_{\varphi}^{2}T/2\right)\exp\left(-\frac{A^{2}}{2(\zeta^{2}+T)}\right).
\]
 If we chose $A_{N}=\sqrt{2(\zeta^{2}+T)\log(N)}$ and $D_{N}=\frac{2e^{2}(\zeta^{2}+T)}{\zeta^{2}}\log(N)$,
then 
\[
\inf_{\varphi\in\mathrm{span}\{z^{n}:~n\leq D_N\}}\frac{1}{T}\int_{0}^{T}\int\left((\mu_{t}\star\varphi)(x)-T_{K_{\varphi}}[\psi\star\mu_{t}](x)\right)^{2}\mu_{t}(x)\,dx\,dt\lesssim\frac{\exp\left(K_{\varphi}^{2}T/2\right)}{N}.
\]
In other words, we obtain an exponential decay of the approximation error for the polynomial basis as well. \qed
\end{rem}

\section{Risk minimization over vector spaces} \label{sec2ab}

\noindent
While the empirical minimization approach presented in the preceding section serves as a valuable estimation method, it is not without its limitations. Firstly, the construction of the estimator $\vphi_N$ in \eqref{defphiN} inherently assumes knowledge of the constant $K_{\vphi}$. Secondly, a significant drawback lies in the necessity of the compactness property of functional spaces $S_m$ with respect to the sup-norm, a condition that is notably stringent.

\subsection{Setting and construction of the estimator} 
This section focuses on a more traditional setting of risk minimization within vector spaces. We assume that the finite-dimensional spaces
$S_N$ are equipped with an inner product $\langle \cdot, \cdot \rangle$,
and $(e_1,\ldots,e_{D_N})$ represents an orthonormal basis of $S_N$ satisfying $\|e_j\|_{\infty}<\infty$ for all $j=1,\ldots, D_N$.
\par
Due to the vector space structure, the estimator $\vphi_N$ introduced at \eqref{defphiN} can be computed explicitly as
\begin{equation}
\label{LSE}
\vphi_{N}=\sum_{j=1}^{D_N} \left(\theta_{N}\right)_j e_j,
\end{equation}
where $\theta_{N}=\Psi_{N}^{-1} Z_{N}$ and the random vectors $\Psi_{N},\ Z_{N}$, taking values in $\R^{D_N \times D_N}$ and $\R^{D_N}$ respectively, are given by 
\begin{align}
(\Psi_{N})_{jk} &= \lan e_j,e_k \ran_N, \nonumber \\[1.5 ex]
\label{defsol}
(Z_{N})_j&=   \frac{1}{NT} \sum_{i=1}^{N} \int_0^T e_j \star  \mu_t^N (X_t^{i,N})
 \, dX_t^{i,N}.
\end{align}
There is a theoretical counterpart to  the empirical matrix $\Psi_{N}$. Indeed, we have that 
$\E{\Psi_{N}}\approx \Psi$ where the matrix $\Psi$ is given by the formula
\begin{equation} \label{Psidef}
(\Psi)_{jk} = \lan e_j,e_k \ran_{\star}.
\end{equation}
Obviously, both matrices are positive semidefinite by construction.
In the next step, we will regularise the  estimator $\vphi_{N}$. For a given sequence  $D_N$, the operator norm $\|\Psi_{N}^{-1}\|_{\text{op}}$ may become too large compared to the sample size $N$ and we would like to avoid such situations. For this purpose, we  restrict   the growth of $\|\Psi^{-1}\|_{\text{op}}.$ 
\begin{assumption}
	\label{assm}
	Let $\eta\geq 5$ be a given number. We assume that $\Psi_{N}$ is invertible almost surely and the sequence $D_N$ is such that the following growth condition is satisfied:
	\[
	L_N^2 \|\Psi^{-1}\|_{\text{op}}^2 \leq \frac{c_{\eta,T} NT}{4 \log(NT)},
	\] 
	where $L_N := \sum_{j=1}^{D_N} \|e_j\|^2_{\infty}$ and $c_{\eta,T}: = (72\eta T)^{-1}$.
\end{assumption}
\par
We now consider the regularised version of the initial estimator $\vphi_{N}$  introduced at \cref{defphiN}:
\begin{equation} \label{finalest}
\wphi_{N} = \vphi_{N} 1_{ \left\{L_N^2 \|\Psi_{N}^{-1}\|_{\text{op}}^2 \leq \frac{c_{\eta,T} NT}{\log(NT)} \right\}}
\end{equation} 
An analogous methodology was suggested in \cite{CG20b} within the framework of classical stochastic differential equations. Nevertheless, the probabilistic analysis of $\wphi_{N}$ becomes notably more intricate due to the sophisticated structure of the model. Moreover, unlike the analysis in \cite{CG20b}, the theoretical quantity $\|\Psi^{-1}\|_{\text{op}}^2$ in Assumption \ref{assm} proves challenging to control due to the absence of information about $\mu_t$. We will discuss this condition in details in Section \ref{secAssu4}.

\subsection{Main results} \label{sec3}

\noindent 
This section is devoted to the asymptotic analysis of the estimator $\wphi_{N}$. 
We start by introducing two random sets, which will be crucial for our proofs. We define
\begin{equation} \label{setdef}
\Lambda_N:= \left\{ L_N^2 \|\Psi_{N}^{-1}\|_{\text{op}}^2 \leq \frac{c_{\eta,T} NT}{\log(NT)} \right\},
\qquad  \Omega_N:=  \left\{ \left| \frac{\|f\|_N^2}{\|f\|_{\star}^2} -1     \right|\leq \frac 12 ~\forall f\in S_N \right\}.
\end{equation}
We recall that $\Lambda_N$ reflects the cut-off introduced in \cref{finalest}. On the other hand, on 
$\Omega_N$ the norms $\|\cdot \|_N$ and $\|\cdot\|_{\star}$ are equivalent, i.e.
\[
\frac{1}{2} \|f\|_{\star}^2 \leq \|f\|_N^2 1_{\Omega_N} \leq \frac{3}{2} \|f\|_{\star}^2. 
\]
Our first theoretical statement shows that, under \cref{assm}, $\P(\Lambda_N)$ and $\P(\Omega_N)$ approach $1$ as $N\to \infty$.  

\begin{prop} \label{prop1}
Suppose that Assumptions \ref{assLip} and \ref{assm} are satisfied. 
\begin{itemize} 
\item[(i)] There exists $k,n \in \N$ such that for all $N\geq n$:
\begin{equation*}
\P \left( \left\| \Psi^{-1/2} \Psi_{N}\Psi^{-1/2} - I_{D_N} \right\|_{\text{\rm op}} >\frac 12  \right)
\leq D_N^{k/4} \exp \left( - \frac{kc_{\eta,T} \eta NT}{16L_N^2 \|\Psi^{-1}\|_{\text{\rm op}}^2} \right),
\end{equation*}
where $I_{D_N}$ denotes the identity matrix in $\R^{D_N \times D_N}$.
\item[(ii)] It holds that 
\[
\P(\Lambda_N^c) \leq \P(\Omega_N^c) \leq (NT)^{-\eta +1}.
\]
\end{itemize}
\end{prop}

\noindent
The concentration bound in \cref{prop1}(i) is key in understanding the asymptotic behaviour of the estimator $ \wphi_{N}$. 
The main result of this section is the following theorem.

\begin{theo} \label{th1}
Suppose that Assumptions \ref{assLip} and \ref{assm} are satisfied. Then we obtain uniform bounds 
\begin{equation} \label{main1}
\sup_{\vphi}
\E\left[\|\wphi_{N} -\vphi \|_N^2\right] \leq  \sup_{\vphi }
\E\left[\inf_{f\in S_N } \|f-\vphi \|_{N}^2 \right]
+ \frac{CD_N }{NT}
\end{equation}
as well as
\begin{align} 
\sup_{\vphi}
\E\left[\|\wphi_{N} -\vphi \|_{\star}^2\right] &\leq \left( 1+ o(1)\right)
\sup_{\vphi} \inf_{f\in S_N } \|f-\vphi \|_{\star}^2 \nonumber \\[1.5 ex]
&+CN^{-1/2}(1+L_N)\sup_{\vphi} \|\vphi_{\star} - \vphi\|_{\star \star}^2 + \frac{CD_N}{NT}, \label{main2}
\end{align}
where  the supremum is taken over the class of functions $\Phi$ defined as
\[
\Phi:= \left\{\vphi:\R \to\R:~\|\vphi\|_{\infty}\leq K_1,~K_2\leq \|\vphi\|_{\text{\rm Lip}}\leq K_{3} \right\}
\]
for fixed constants $K_1,K_2,K_3>0$. Here the norm $\|f\|_{\star \star}$ is defined as
\[
\|f\|_{\star \star}^2:=\frac{1}{T} \int_0^T \int_{\R} f^2 \star  \mu_t (x)
  \mu_t(x)\, dx\, dt.
\]
\end{theo}

\noindent 
We remark that $\|f\|_{ \star} \leq \|f\|_{\star \star}$. However, it does not hold a priori that $\|\vphi_{\star} - \vphi\|_{\star \star}\to 0$ as $N\to \infty$. We will discuss this term in the next subsection.

\subsection{Checking Assumption \ref{assm} and bounding the norm
 $\|\vphi_{\star} - \vphi\|_{\star \star}$} \label{secAssu4}
The primary challenge in implementing the theoretical findings from the preceding section lies in selecting an appropriate dimension $D_N$, which ensures the fulfillment of Assumption 4. This requirement is essentially equivalent to determining an upper bound for the quantity $\|\Psi^{-1}\|_\opr$. In this context, we introduce an approach to establish an upper bound on the operator norm $\|\Psi^{-1}\|_\opr$, guided by a Gaussian-type condition on the densities $\mu_t$.
We introduce the following condition on $(\mu_t)_{t\in [0,T]}$:
\begin{assumption}
	\label{lowecond}
	We suppose that the densities $\mu_t$ are symmetric and there exist numbers $c_1,c_2>0$ and functions $g_2\in L^2([0,T])$, $g_1 g^2_2 \in L^1([0,T])$ such that 
 \[
 \mu_t(x)\geq g_1(t) \exp(-x^2/c_1) \qquad \text{and} \qquad
 \widehat{\mu}_t(x)\geq g_2(t) \exp(-x^2/c_2).
 \]
\end{assumption}
\begin{rem}
The lower bound for the densities $(\mu_t)_{t\in [0,T]}$ can be deduced via
\cite[Theorem~1]{QZ02}. Indeed, under Assumption \ref{assmu0}, it holds that
\begin{align*}
    \mu_t(x) & \ge \frac{1}{\sqrt{2\pi t}} \int_\R \mu_0(y) \int_{\frac{\abs{x-y}}{\sqrt t}}^\infty z \exp\left(-\frac{(z+K_\varphi \sqrt t)^2}{2}\right)\, dz \, dy\\
    & \ge \frac{\exp\left(-K_\varphi^2 t\right)}{2\pi\sqrt{t\zeta^2}} \int_\R \exp\left(-\frac{y^2}{2\zeta^2}\right) \exp\left(-\frac{(x-y)^2}{t}\right)\, dy\\
    & = \frac{\exp\left(-K_\varphi^2 t\right)}{\sqrt{2\pi (\zeta^2 + t/2)}} \exp\left(-\frac{x^2}{2\zeta^2 + t}\right).
\end{align*}
Obtaining lower bounds for the Fourier transforms $(\widehat{\mu}_t)_{t\in [0,T]}$ is a more delicate problem. Some partial results in this direction have been demonstrated in \cite[Theorem~5.4]{ABPPZ23} in the ergodic setting. 
\qed
\end{rem}
\noindent
We consider the vector space $S_N$ generated by functions 
$(e_k)_{1\leq k \leq D_N}$ with 
$$e_k(x)=(2A_N)^{-1/2}
\exp(i\pi kx/A_N), \qquad x\in \R,$$ 
which is an orthonormal system in $L^2([-A_N,A_N])$. For this choice of basis we obtain the identity 
\begin{align} \label{Psiidenti}
\Psi_{kl}= \langle e_k,e_l \rangle_{\star}= 
\frac{1}{2A_NT} \int_0^T \widehat{\mu}_t(-k/A_N) \widehat{\mu}_t(l/A_N) \widehat{\mu}_t((l-k)/A_N) dt.
\end{align}
We deduce the following result.

\begin{prop} \label{propPsi}
Suppose that Assumption \ref{lowecond} holds. Assume that $D_N=a_{D} \log(N)$ and $A_N=a_A \sqrt{\log(N)}$ for some $a_D,a_A>0$. Then it holds that 
\begin{align*}
\|\Psi^{-1}\|_\opr \leq v_T N^{\frac{2a_D^2}{c_2 a_A^2} + \frac{a_A^2}{c_1}} (1+o(1)),
\end{align*}
where $v_T:=2\sqrt{2\pi} T \left(\int_{0}^T
g_1(t) g_2(t)^2 dt \right)^{-1}$. 
\end{prop}

\noindent
In the next step we will study the norm $\|\vphi_{\star} - \vphi\|_{\star \star}$. For this purpose, we assume that $\vphi$ has the representation
\[
\vphi(x)= \sum_{k=1}^{\infty} c_k e_k(x)
\]
with coefficients $c_k=c_k(N)$ satisfying the condition $\sum_{k=1}^{\infty} |c_k|<\infty$.

\begin{prop} \label{doublestar}
Suppose that Assumption \ref{assmu0} is satisfied. Then it holds that 
\[
\|\vphi_{\star} - \vphi\|_{\star \star}^2 \lesssim \left(1+
 A_N^{-1} D_N^4
\|\Psi^{-1}\|_{\text{\rm op}}^2 \exp\left(-\frac{3D_N^2\zeta^2}{2A_N^2}\right)\right) \left(\sum_{k=D_N+1}^{\infty} |c_k|\right)^2 .
\]
\end{prop}

\noindent
Now, we can combine the statements of Theorem  \ref{th1}, Proposition 
\ref{propPsi} and Proposition \ref{doublestar},  to derive the convergence rate for  $\|\wphi_{N} -\vphi \|_{\star}$. 

For clarity, we focus on the ergodic scenario where
 $\mu_t=\mu_0$ and $\mu_0$ satisfies Assumption \ref{assmu0}. In this context,
the constants $c_1$ and $c_2$ from Assumption \ref{lowecond}  are explicitly given by
\[
c_1= 2\zeta^2, \qquad c_2= \frac{2}{\zeta^2}.
\]
As $\|e_j\|_{\infty}^2 = A_N^{-1}$ and $D_N=a_{D} \log(N)$, $A_N=a_A \sqrt{\log(N)}$, Proposition  \ref{propPsi} implies the following bound:
\[
L_N^2 \|\Psi^{-1}\|_{\text{op}}^2 \lesssim \frac{a_D^2}{a_A^2} \log(N)
N^{\frac{2a_D^2\zeta^2}{a_A^2} + \frac{a_A^2}{\zeta^2}}.
\]
Hence, for Assumption \ref{assm} to hold true, the constants $a_A$ and $a_D$ need to fulfill 
\[
\frac{2a_D^2\zeta^2}{a_A^2} + \frac{a_A^2}{\zeta^2}<1.
\]
On the other hand, exploring the proof of Theorem
\ref{approxerr}, the approximation error is obtained as 
\[
\|\vphi_{\star} - \vphi\|_{ \star}^2 \lesssim N^{-\frac{a_D^2\zeta^2}{a_A^2}}
+N^{-\frac{a_A^2}{2\zeta^2}}.
\]
Consequently, if we choose $a_D^2\zeta^2/a_A^2= a_A^2/(2\zeta^2)=(1-\epsilon)/4$ for a small $\epsilon >0$, we finally conclude that 
\begin{align}
\E\left[\|\wphi_{N} -\vphi \|_{\star}^2 \right]\lesssim N^{-1/4 +\epsilon}
\end{align}
provided the condition $(\sum_{k=D_N+1}^{\infty} |c_k|)^2\leq CN^{-z}$ with
$z>5/8$ holds. 
The substantial decrease in the convergence rate, as observed compared to Theorem \ref{approxerr}, directly stems from the stringent constraint imposed by Assumption \ref{assm}. Still this is the first result in  literature regarding the convergence rates of the linear-type estimates within the context of McKean-Vlasov Stochastic Differential Equations.

\section{Lower bounds} \label{sec.low}

\noindent
This section is dedicated to deriving minimax lower bounds. We start by establishing a lower bound for the previously examined estimation problem concerning the norm $\|\cdot\|_{\star}$. To achieve this, we focus on a simplified scenario involving i.i.d.~observations drawn from a McKean-Vlasov SDE:
\begin{equation*} 
\left\{
\begin{array}{ll}
	dX_t =  (\vphi \star \mu_t )(X_t)\, dt + \sigma dB_t,\\ \textnormal{Law}(X_t)=\mu_t,
\end{array}
\right.
\end{equation*}
and denote by $\P_{\vphi}$ the associated probability measure. We assume that $X_0^{\vphi}=x \in \R$ for all $\vphi$. 
The ensuing theorem establishes a minimax lower bound with respect to  the norm $\|\cdot\|_{\star}$.

\begin{theo} \label{Lower1}
There exists a constant $c>0$ such that, for every $N\in\N$,
\begin{equation}
		\inf_{\widehat\varphi_N}\sup_{\varphi}\P^{\bigotimes N}_\varphi\left(
		\|\widehat \varphi_N-\varphi\|_{\star}>cN^{-1/2}\right)>0,
	\end{equation}
	where $\P^{\bigotimes N}_\varphi$ is the $N$-fold product measure, the supremum is taken over all functions
 $\vphi$ satisfying Assumption \ref{assphi} and $\|\vphi\|_\infty\leq (32T)^{-1/2}$, and the infimum is taken over all estimators of $\varphi$ retrieved from $N$ observations of $(X_t)_{t\in[0,T]}$.
\end{theo}

\begin{proof}
Consider two functions $\vphi_0,\vphi_1 \in S$ and denote by $\mu_t^0$ and $\mu_t^1$ the corresponding marginal densities. The Kullback-Leibler divergence $\text{KL}(\P_{\vphi_0}, \P_{\vphi_1})$ between the probability measures 
$\P_{\vphi_0}$ and $\P_{\vphi_1}$ can be explicitly computed as
\begin{align*}
\text{KL}(\P_{\vphi_0}, \P_{\vphi_1}) & =\frac{1}{2}\int_{0}^{T}\int_{\mathbb{R}}\left(\varphi_{0}\star\mu_{t}^{0}(x)-\varphi_{1}\star\mu_{t}^{1}(x)\right)^{2}\mu_{t}^{0}(x)\,dx\,dt.
\end{align*}
As a consequence we obtain the inequality 
\begin{align*}
\text{KL}(\P_{\vphi_0}, \P_{\vphi_1}) & \leq T\|\varphi_{0}-\varphi_{1}\|_{\star}^{2}+\|\varphi_{1}\|_{\infty}^{2}\int_{0}^{T}\left[\int_{\mathbb{R}}\left|\mu_{t}^{0}(x)-\mu_{t}^{1}(x)\right|\,dx\right]^{2}\,dt,
\end{align*}
where the $\|\cdot\|_{\star}$-norm is computed with respect to $\mu_t^0$. 
Furthermore, using Theorem~1.1 of \cite{BRS16}, we derive
\begin{align*}
&\int_{0}^{T}\left[\int_{\mathbb{R}}\left|\mu_{t}^{0}(x)-\mu_{t}^{1}(x)\right|\,dx\right]^{2} dt
\\[1.5 ex]
& \leq4\int_{0}^{T}\left[\int_{0}^{t}\int(\varphi_{0}\star\mu_{s}^{0}(x)-\varphi_{1}\star\mu_{s}^{1}(x))^{2}\mu_{s}^{0}(x)\,dx\,ds\right]\,dt
 \\[1.5 ex]
 & \leq8T\text{KL}(\P_{\vphi_0}, \P_{\vphi_1}).
\end{align*}
Since $\|\varphi_{1}\|_{\infty}^{2}\leq(16T)^{-1}$ we finally deduce that  
\[
\text{KL}(\P_{\vphi_0}, \P_{\vphi_1})\leq2T\|\varphi_{0}-\varphi_{1}\|_{\star}^{2}.
\]
Applying the two hypotheses method described in \cite[Theorem 2.2]{T09}, we obtain the statement of the theorem. 
\end{proof}

We will now consider the $L^2(\R)$-norm instead of $\|\cdot\|_{\star}$ and compare the results to 
\cite{BPP22}. We recall that logarithmic rates have been obtained in the $L^2(\R)$-norm  in \cite[Theorem 5.1]{BPP22} for a similar estimation problem given the observation scheme $(X_T^{i,N})_{1\leq i \leq N}$ with $N,T\to \infty$. It turns out that these rates cannot be improved to polynomial ones in their framework, and the same holds in our case (note however that polynomial rates may appear under a different set of conditions as it has been demonstrated in \cite{ABPPZ23}). To see this, we again consider simplified setting of i.i.d.~observations drawn from McKean-Vlasov SDE
\cref{eq:McKeanEq}. Furthermore, we consider a similar model as presented in \cite{BPP22}: Let $W:\R \to \R$ be an even and strictly convex interaction potential of the form
\[
W(x) = p(x) + \beta(x), \qquad p(x) = \sum_{j=1}^J a_j x^{2j},
\]   
where $(a_j)_{j=1,\dots,J}$ are known constants and $\beta\in C_c^2(\R)$. The potential $W$ determines the interaction function $\vphi$ via $\vphi=W'$ and we are interested in the estimation of the non-parametric part $\beta$ in the case of stationarity. The invariant law $\mu$ solves the elliptic nonlinear Fokker-Planck equation
\begin{equation}
	\left[\frac12\frac{\d^2}{\d x^2}+\frac\d{\d x}\left(\varphi\star\mu\right)\right]\mu=0,
\end{equation}
meaning that $\mu$ is given by the implicit equation
\begin{equation}
	\mu(x)=c_\mu\exp\left(-W\star\mu(x)\right),
\end{equation}
where $c_\mu$ is a normalising constant such that $\mu$ is a probability density. 

To introduce the appropriate class of functions, we consider $C_c^2(\R)$-wavelets. More specifically, we let the mother wavelet be a symmetric function
$\psi \in C_c^2(\R)$ with $\text{supp} (\psi) = [-1/2,1/2]$ and, for $n,k \in \Z$, define  
\[
\psi_{n,k}(x) = 2^{n/2} \psi(2^nx-k).
\]
The set $(\psi_{n,k})_{n,k \in \Z}$ forms a basis of $L^2(\R)$. Now, for $\al,\ K>0$, we consider the functional space
\begin{equation*}
	\mathcal F^\alpha_K=\left\{\beta^\prime: ~\beta\in C^2_c(\R),\ \sum_{n=1}^\infty\sum_{|k|>m}\lan \beta^\prime, \psi_{n,k}\ran_{L^2(\R)}^2\leq K m^{-\alpha}\quad \forall m\in\N\right\}.
\end{equation*}
For the estimation of the nonparametric part $\beta'$ of $\vphi$ based on i.i.d.\ observations of \cref{eq:McKeanEq} we obtain the following lower bound. 

\begin{theo} \label{minimax}
Assume that the coefficients $(a_j)_{1\leq j\leq J}$ and the function
$\beta$ satisfy the  conditions of \cite[Theorem 5.1]{BPP22}. Then there exists a constant $c>0$ such that, for every $N\in\N$,
\begin{equation}
		\inf_{\widehat\beta'_N}\sup_{\beta'\in \mathcal F^\alpha_K}\P^{\bigotimes N}_\varphi\left(
		\|\widehat \beta'_N-\beta'\|_{L^2(\R)}^2>c\left(\log N\right)^{-\frac \alpha{4J}}\right)>0,
	\end{equation}
	where $\P^{\bigotimes N}_\varphi$ is the $N$-fold product measure and the infimum is taken over all estimators of $\beta'$ retrieved from $N$ observations of $(X_t)_{t\in[0,T]}$.
\end{theo}

\begin{proof}
We will use the two hypotheses method described in \cite[Theorem 2.2]{T09}. We will find two hypotheses of interaction forces $\varphi_0,\varphi_1\in \mathcal F^\alpha_K$ such that, for some constant $c>0$,
	\begin{equation*}
		\norm{\varphi_0-\varphi_1}_{L^2(\R)}^2=c\left(\log N\right)^{-\frac \alpha{4J}}\quad\text{and}\quad\text{KL}\left(\P^{\bigotimes N}_{\varphi_0}, \P^{\bigotimes N}_{\varphi_1}\right)\lesssim 1.
	\end{equation*}
	The Radon-Nikodym density of $\d \P_{\varphi_1}/\d \P_{\varphi_0}$ can be computed using Girsanov's theorem: denoting the stationary marginals as $\mu_{i}$, $i=0,1$, we have for $X\in C([0,T];\R)$
	\begin{align*}
		\log\left(\frac{\d\P_{\varphi_1}}{\d\P_{\varphi_0}}(X)\right)&=\int_0^T \varphi_1\star\mu_1(X_t)-\varphi_0\star\mu_0(X_t)\d X_t
  \\
  &-\frac12\int_0^T\varphi_1\star\mu_1(X_t)^2-\varphi_0\star\mu_0(X_t)^2\d t,
	\end{align*}
	such that
	\begin{align*}
		\text{KL}\left(\P_{\varphi_0}, \P_{\varphi_1}\right)&=	\int_\R\log\left(\frac{\d\P_{\varphi_1}}{\d\P_{\varphi_0}}(x)\right)\mu_1(x)\d x
  \\
  &=\frac {T}2\int_\R\left(\varphi_1\star\mu_1(x)-\varphi_0\star\mu_0(x)\right)^2\mu_1(x)\d x.
	\end{align*}
	In a similar manner as in \cite[Theorem 5.1]{BPP22}, we can show that
	\begin{equation*}
		\text{KL}\left(\P_{\varphi_1},\P_{\varphi_0}\right) \le \gamma^2 \text{KL}\left(\P_{\varphi_1},\P_{\varphi_0}\right) + \frac{T}{2} \int_\R (\beta^\prime_0-\beta^\prime_1)\star\mu_0(x)^2\mu_1(x)\d x,
	\end{equation*}
    with $\gamma \in (0,1)$. So, it remains to show the estimate
	\begin{equation}
		\int_\R (\beta^\prime_0-\beta^\prime_1)\star\mu_0(x)^2\mu_1(x)\d x\lesssim\frac1N.\label{eq:twohyp}
	\end{equation}
	Now we will choose the two hypotheses. Let $\rho>0$, $M\in\N$ be rescaling parameters to be determined later. For $n\in\N$, we write $l_n:=\lfloor \log_2n\rfloor$ and define
	\begin{equation*}
		f=\rho\sum_{\abs n,\abs k=M}^{2M}\psi_{l_n,k}.		
	\end{equation*}
	Note that our choice of basis elements implies that $\operatorname{supp}f\subset[-2M^2,2M^2]$ and $\norm f_{L^2(\R)}^2=\rho^2M^2$. We will consider the hypotheses $\varphi_0\in \mathcal F^\alpha_K$ and $\varphi_1=\varphi_0+f$. In order for $\varphi_1\in \mathcal F^\alpha_K$ to hold, we need $d(f,S_m)\le Km^{-\alpha}$ for all $m\in\N$. There are two cases to consider for this, namely
	\begin{equation*}
		\rho^2\sum_{n_1=1}^\infty\sum_{\abs {k_1}>m}\sum_{\abs{n_2},\abs{k_2}=M}^{2M}\sp{\psi_{n_1,k_1}}{\psi_{{l_{n_2}},k_2}}_{2}^2
		\begin{cases}
			\lesssim\norm f_{2}^2=\rho^2M^2, & m\le l_{2M},\\
			=0, & m>l_{2M}.
		\end{cases}
	\end{equation*}
	Hence, we require the condition $\rho^2M^2\lesssim \lfloor \log_2(2M)\rfloor^{-\alpha}$. We enforce a stricter condition
	\begin{equation*}
		\rho^2=O(M^{-\alpha-2}).
	\end{equation*}
	This means we choose
	\begin{equation*}
		\norm{\varphi_0-\varphi_1}_{L^2(\R)}^2=\norm{f}_{L^2(\R)}^2=\rho^2M^2=M^{-\alpha}.
	\end{equation*}
	We will now verify that \cref{eq:twohyp} holds. The main ingredient is the observation that, for $i=0,1$,
	\begin{equation}
		\mu_i(x)\lesssim \exp\left(-a_Jx^{2J}\right),\label{mutail}
	\end{equation}
	which in particular implies $\|\mu_i\|_{L^2(\R)}<\infty$. By symmetry of $\mu_i$ and $f$, it suffices to show
	\begin{equation}
		\int_0^\infty f\star\mu_0(x)^2\mu_1(x)\d x\lesssim\frac1N.\label{KL}
	\end{equation}
	Let us split the integral at some $k>0$ which we determine later. The tail behaviour is governed by \cref{mutail}: Using Cauchy-Schwarz, we have $\|f\star\mu_0\|_\infty^2\le\|f\|_{L^2(\R)}^2\|\mu_0\|_{L^2(\R)}^2\lesssim\rho^2M^2$, such that
	\begin{equation*}
		\int_k^\infty f\star\mu_0(x)^2\mu_1(x)\d x\lesssim\rho^2M^2\frac{\exp\left(-k^{2J}\right)}{k^{2J-1}}.
	\end{equation*}
	For the remaining integral, we bound $\mu_1$ such that
	\begin{equation*}
		\int_0^kf\star\mu_0(x)^2\mu_1(x)\lesssim k\sup_{x\in[0,k]}f\star\mu_0(x)^2.
	\end{equation*}
	We recall that $\operatorname{supp}f=[-2M^2,2M^2]$, which implies
	\begin{align*}
		\sup_{x\in[0,k]}\int_{-\infty}^{x-2M^2}f(x-y)\mu_0(y)\d y 
  &\le \norm{f}_{L^2(\R)}^2\int_{-\infty}^{k-2M^2}\mu_0(y)^2\d y
  \\
  &\lesssim \rho^2M^2\frac{\exp\left(-(2M^2-k)^{2J}\right)}{(2M^2-k)^{2J-1}}.
	\end{align*}
	From this, we see that by choosing $k=M^2$, the entire integral is indeed bounded by
	\begin{equation*}
		\int_0^\infty f\star\mu_0(x)^2\mu_1(x)\d x\lesssim \exp\left(-M^{4J}\right),
	\end{equation*}
	such that choosing $M=c\left(\log N\right)^{\frac 1{4J}}$ with some constant $c>0$ gives \cref{KL}, which in turn gives us the rate
	\begin{equation*}
		\norm{\varphi_0-\varphi_1}_{L^2(\R)}^2=CM^{-\alpha}=C\left(\log N\right)^{-\frac \alpha{4J}},
	\end{equation*}
	which implies the claim of the theorem.
\end{proof}

\section{Proofs} \label{sec5.proofs}

\noindent
This section is devoted to proofs of the main results. We will 
often transfer probabilistic results from i.i.d.~ observations
of the McKean-Vlasov SDE displayed in \cref{eq:McKeanEq}
to observations of the original particle system at \cref{model} via a change of measure device established in \cite{DM22a}. For this purpose we consider two $N$-dimensional processes given by the weak solutions of
\[
\begin{array}{cc}
d\overline{X}_{t}^{i} & =(\varphi\star\mu_{t})(\overline{X}_{t}^{i})+\sigma dW_{t}^{i}\end{array},\quad1\leq i\leq N,
\] 
\[
\begin{array}{cc}
dX_{t}^{i} & =(\varphi\star\mu_{t}^{N})(X_{t}^{i})+\sigma dW_{t}^{i}\end{array},\quad1\leq i\leq N
\]
having the same initial value, and denote by $\overline{\P}^{N}$ and $\P^{N}$, respectively, the associated probability measures.

\subsection{Proof of Theorem \ref{Th2.3}} \label{secPr2.3}
In the first part of the proof we consider the setting of i.i.d.~observations drawn from the mean field limit 
\cref{eq:McKeanEq}, that is we work under 
$\overline{\P}^{N}$. 
We recall the inequality 
\begin{multline}
\|\vp_N -\vp\|_{\star}^2 \leq \|\vp_{\star} -\vp\|_{\star}^2 
+ 2\left(\sup_{f\in S_N} \left| \|f-\vp\|_{N}^2 - \|f -\vp\|_{\star}^2  \right| +  \sup_{f\in S_N} 
|\ga'_N(f)| \right),
\end{multline}
where $\vp_{\star}$ has been introduced in \cref{defphiN}.
Next, we set  $\ga'_N(f)=:2M_N(f)/\sqrt{NT}$ where for all $f\in S$, \(M_N(f)\) is a martingale with quadratic variation 
\[
\frac{\sigma^2}{NT} \sum_{i=1}^N \int_0^T [(f\star \mu_t^N)(X_t^{i,N})]^2 dt \leq \sigma^2 \|f\|^2_{\infty}.
\]
Hence, we deduce by the Bernstein inequality for continuous martingales, 
\[
\overline{\P}^N(|M_N(f) - M_N(g)|\geq a) \leq 2 \exp\left(-a^2/(\sigma^2\|f-g\|^2_{\infty})\right)
\]
for any \(a>0\) and \(f,g \in S_N\). Since for all $q\ge 1$, the $L^q(\Omega)$-norm is dominated by $\norm\cdot_{\psi_{e,2}}$, an application of the general maximal inequality \cite[Theorem~8.4]{KO8} implies
\begin{align} \label{ineqa1}
\Bigl\|\sup_{f\in S_N}\left|M_N(f) -M_N(f_0)\right|\Bigr\|_{L^2(\Omega)}\lesssim \mathsf{DI}(S_N,\|\cdot\|_\infty,\psi_{e,2})
\end{align}
for any $f_0\in S_N$. Consider now the difference 
$ \|g\|_{N}^2 - \|g\|_{\star}^2$ for a function $g\in S_N(2K_\vphi)$. We obtain the decomposition
\begin{align} \label{Ustati}
\|g\|_{N}^2 - \|g\|_{\star}^2 = U_{1,N}(g) + U_{2,N}(g) +U_{3,N}(g) + O(N^{-1})
\end{align}
where \(\De_t(x, y):=g(x -y) - (g\star \mu_t )(x),\) and
\begin{align*}
U_{1,N}(g) &:= \frac{2}{T} \int_0^T \int_{\R} \left(\frac 1N \sum_{j=1}^N \De_t(x, X_t^{j,N})\right)
(g\star \mu_t)(x) \mu_t(x) dx dt \\[1.5 ex]
&+\left( \frac{1}{NT} \sum_{i=1}^N \int_0^T (g\star \mu_t)(X_t^{i,N})^2 dt - \frac{1}{T} \int_0^T \int_{\R}
(g\star \mu_t)(x)^2 \mu_t(x)dxdt \right), 
\end{align*}
\begin{align*}
U_{2,N}(g) &:= \frac{1}{N(N-1)(N-2)} \sum_{i=1}^N  \sum_{j,k\not=i, j\not=k} \frac 1T \int_0^T
\int_{\R} \De_t(x, X_t^{j,N}) \De_t(x, X_t^{k,N}) \mu_t(x) dx dt \\[1.5 ex]
&+ \frac{2}{NT} \sum_{i=1}^N \int_0^T \frac{1}{N-1} \sum_{j\not = i} \De_t(X_t^{i,N},X_t^{j,N})
(g\star \mu_t)(X_t^{i,N}) dt \\[1.5 ex]
&- \frac{2}{NT} \sum_{i=1}^N \int_0^T \frac{1}{N-1} \sum_{j\not = i}\int_{\R} \De_t(x,X_t^{j,N})
(g\star \mu_t)(x) \mu_t(x)\, dx\, dt,  
\end{align*}
\begin{align*}
U_{3,N}(g) &:= \frac{1}{N(N-1)(N-2)} \sum_{i=1}^N  \sum_{j,k\not=i, j\not=k} \frac 1T \int_0^T
\De_t(X_t^{i,N},X_t^{j,N}) \De_t(X_t^{i,N},X_t^{k,N})\, dt \\[1.5 ex]
&-\frac{1}{N(N-1)(N-2)} \sum_{i=1}^N  \sum_{j,k\not=i, j\not=k} \frac 1T \int_0^T
\int_{\R} \De_t(x,X_t^{j,N}) \De_t(x,X_t^{k,N}) \mu_t(x)\, dx \, dt.
\end{align*}
The key observation is that each $U_{r,N}$ is a degenerate $U$-statistics of order $r$ for $r=1,2,3$, while
$U_{1,N}$ is the leading term. 
First,  turn to \(U_{1,N}\) and note that 
\begin{eqnarray*}
U_{1,N}(g)=\frac{1}{N}\sum_{j=1}^N \left(Z_j(g)-\E_{\overline{\P}^N} Z_j(g)\right)
\end{eqnarray*}
with 
\begin{eqnarray*}
Z_j(g)=\frac{1}{T} \int_0^T \left(2\int_{\R} \De_t(x, X_t^{j,N})
(g\star \mu_t)(x) \mu_t(x)\, dx\right) + (g\star \mu_t)(X_t^{i,N})^2\, dt.
\end{eqnarray*}
It holds that
\[
|Z_j(g)|\leq 8K_\vphi^2,\quad |Z_j(g)-Z_j(g')|\leq 4K_\vphi\|g-g'\|_\infty, 
\]
with probability \(1\) for \(g,g'\in S_N(2K_\vphi).\)  Hence,
\[
\|Z_j(g)-Z_j(g')\|_{\psi_{e,2}}\lesssim K_\vphi\|g-g'\|_\infty.
\]
This implies that the process \(\widetilde U_{1,N}=\sqrt{N} U_{1,N}\) has subgaussian increments  and  
\begin{eqnarray*}
\|\widetilde U_{1,N}(g)-\widetilde U_{1,N}(g')\|_{\psi_{e,2}}\lesssim K_\vphi\|g-g'\|_\infty.
\end{eqnarray*}
 Fix
some $g_0\in S_N(2K_\vphi)$. By the triangle inequality,
\[
\sup_{g\in S_N(2K_\vphi)}|\widetilde U_{1,N}(g)|\le\sup_{g,g'\in S_N(2K_\vphi)}| \widetilde U_{1,N}(g)-\widetilde U_{1,N}(g')|+|\widetilde U_{1,N}(g_0)|.
\]
By the Dudley integral inequality,  see  \cite[Theorem
8.1.6]{Vershynin}, for any $\delta\in(0,1)$, we conclude that 
\[
\sup_{g,g'\in S_N(2K_\vphi)}|\widetilde U_{1,N}(g)-\widetilde U_{1,N}(g')|\lesssim K_\vphi\bigl[\mathsf{DI}(S_N,\|\cdot\|_\infty,\psi_{e,2}) +\sqrt{\log(2/\delta)}\bigr]
\]
holds with probability at least $1-\delta$.  Applying Hoeffding's inequality,
see e.g. \cite[Theorem 2.6.2.]{Vershynin}, we have for any
$\delta\in(0,1)$,
\[
|\widetilde U_{1,N}(g_0)|\lesssim K_\vphi\sqrt{\log(1/\delta)},
\]
with probability at least $1-\delta$. In other words, for any $u>0$, 
\begin{equation*}
    \overline{\mathbb P}^N \left( \sup_{g\in S_N(2K_\vphi)}|\widetilde U_{1,N}(g)| > u \right) \le \exp\left( -\frac{(u-K_\varphi \mathsf{DI}(S_N(2K_\vphi),\|\cdot\|_\infty,\psi_{e,2}))^2}{2 K_\varphi^2} \right).
\end{equation*}
Putting things together we conclude that
\begin{align} \label{ineqa2}
\Bigl\{\E_{\overline{\P}^N}\Bigl[\sup_{g\in S_N(2K_\vphi)}|U_{1,N}(g)|^q\Bigr]\Bigr\}^{1/q} \lesssim \frac{q K_\vphi \, \mathsf{DI}(S_N(2K_\vphi),\|\cdot\|_\infty,\psi_{e,2}) }{\sqrt{N}}.
\end{align}
Furthermore, due to Theorem~\ref{thm:ustat-max},
\begin{align} \label{ineqa3}
 \Bigl\{\E_{\overline{\P}^N}\Bigl[\sup_{g\in S_N(2K_\vphi)}\left|NU_{2,N}(g)\right|^{q}\Bigr]\Bigr\}^{1/q}  \lesssim p^2 K_\vphi\,  \mathsf{DI}(S_N(2K_\vphi),\|\cdot\|_\infty,\psi_{p/2})
\end{align}
and
\begin{align} \label{ineqa44}
 \Bigl\{\E_{\overline{\P}^N}\Bigl[\sup_{g\in S_N(2K_\vphi)}\left|N^{3/2}U_{3,N}(g)\right|^{q}\Bigr]\Bigr\}^{1/q}  \lesssim p^3 K_\vphi\,  \mathsf{DI}(S_N(2K_\vphi),\|\cdot\|_\infty,\psi_{p/3})
\end{align}
for any \(p>2\) and \(q\leq p.\) This implies the statement of Theorem \ref{Th2.3}
in the setting of i.i.d.~observations drawn from the McKean-Vlasov SDE
\cref{eq:McKeanEq}. 

Now, we will transfer the result from i.i.d.~observations to observations of the original particle system at \cref{model} via a change of measure device established in \cite{DM22a}. 
Recalling the definition of probability measures $\overline{\P}^{N}$ and $\P^{N}$, we deduce the identity
\[
\frac{\d\P^{N}}{\d\overline{\P}^{N}}=\mathcal{E}_{T}(\overline{M}^{N})
\]
where $\mathcal{E}_{t}(\overline{M}^{N})=\exp\bigl(\overline{M}_{t}^{N}-\frac{1}{2}\langle \overline{M}^{N}\rangle _{t}\bigr)$
and 
\[
\overline{M}_{t}=\sum_{i=1}^{N}\int_{0}^{t}\sigma^{-1}\left[\varphi\star(\mu_{t}^{N}-\mu_{t})(\overline X_{t}^{i})\right]\,dW_{t}^{i}
\]
is a $\overline{\P}^{N}$-local martingale. According to  \cite[Proposition~19]{DM22a}, we obtain the bound
\begin{equation}
\sup_{N\geq1}\sup_{t\in[0,T-\delta]}\mathbb{E}_{\overline{\P}^{N}}\left[\exp\left(\tau\left(\langle \overline{M}^{N}\rangle _{t+\delta}-\langle \overline{M}^{N}\rangle _{t}\right)\right)\right]\leq C_{\delta,\tau}\label{eq:exp-M}
\end{equation}
for every \(\tau>0,\) $0\leq\delta\leq\delta_{0}$ and some constant \(C_{\delta,\tau}>0\). Let $\xi$ be any $\mathcal{F}_{T}$-measurable
random variable. Then by the H\"older inequality
\begin{align} \label{changeofmeasu}
\mathbb{E}_{\P^{N}}[|\xi|^{r}]=\mathbb{E}_{\overline{\P}^{N}}[|\xi|^{r}\mathcal{E}_{T}(\overline{M}^{N})]\leq\left\{ \mathbb{E}_{\overline{\P}^{N}}[|\xi|^{2r}]\right\} ^{1/2}\left\{ \mathbb{E}_{\overline{\P}^{N}}[\mathcal{E}_{T}^{2}(\overline{M}^{N})]\right\} ^{1/2}.
\end{align}
Next, fix a grid \(0=t_0<t_1<\ldots<t_K=T\) with \(|t_{k+1}-t_k|\leq \delta \) and decompose 
\begin{eqnarray*}
\mathbb{E}_{\overline{\P}^{N}}[\mathcal{E}_{T}^{2}(\overline{M}^{N})]&=&\mathbb{E}_{\overline{\P}^{N}}\prod_{k=1}^K \exp\Bigl(2(\overline{M}_{t_k}^{N}-\overline{M}_{t_{k-1}}^{N})+\langle \overline{M}^{N}\rangle _{t_{k-1}}-\langle \overline{M}^{N}\rangle _{t_k}\Bigr)
\\
&=&\mathbb{E}_{\overline{\P}^{N}}\prod_{k=1}^K \mathcal{E}_{t_k}(\overline{M}_{t_k}^{N}-\overline{M}_{t_{k-1}}^{N})\exp\Bigl(\langle\overline{M}_{t_k}^{N}-\overline{M}_{t_{k-1}}^{N}\rangle
\\
&+&\langle \overline{M}^{N}\rangle _{t_{k-1}}-\langle \overline{M}^{N}\rangle _{t_k}\Bigr)
\end{eqnarray*}
By using martingale property of \(\mathcal{E}_{t_k},\) repeatedly applying the H\"older inequality and
using  (\ref{eq:exp-M}), we deduce the inequality 
\[
\mathbb{E}_{\overline{\P}^{N}}[\mathcal{E}_{T}^{2}(\overline{M}^{N})]\leq C_{\delta,T}.
\]
As a consequence the bounds \cref{ineqa1}-\cref{ineqa44} remain valid under the probability measure $\P^N$. This completes the proof.

\subsection{Proof of Theorem \ref{approxerr}}

\noindent 
Note that $\alpha_{n_{N}-1}< A_N\leq \alpha_{n_{N}}$ and hence $A_N/\alpha_{n_{N}}\to 1$ as $N\to \infty.$  Under Assumption \ref{assphi} considered for the sequence $\alpha_{n_{N}}$, we  obtain the inequality
\begin{align*}
\inf_{g\in S_N}& \frac1T \int_0^T\int_{\R}\left((\varphi\star\mu_{t})(x)-(g\star\mu_{t})(x)\right)^{2}\mu_{t}(x)\,dx\,dt\\
&\leq \frac{1}{\alpha_{n_{N}}T} \int_0^T\|\mu_t\|_{[-\alpha_{n_{N}},\alpha_{n_{N}}]}\int_{-\alpha_{n_{N}}}^{\alpha_{n_{N}}}\left|\sum_{k=D_N+1}^{\infty}c_{k}\exp(i\pi kx/\alpha_{n_{N}})\widehat{\mu}_{t}(-k/\alpha_{n_{N}})\right|^{2}\,dx\,dt  +R(\alpha_{n_{N}})\\
&\leq \frac{1}{\alpha_{n_{N}}T} \int_0^T \|\mu_t\|_{[-\alpha_{n_{N}},\alpha_{n_{N}}]} \sum_{k=D_N+1}^{\infty}\left(|c_{k}||\widehat{\mu}_{t}(-k/\alpha_{n_{N}})|\right)^{2}\,dt  +R(\alpha_{n_{N}})
\end{align*}
where 
\begin{align*}
R(A)  
  \leq \frac{4K_\vphi^{2}}T\int_0^T\int_{|x|>A}\mu_{t}(x)\,dx\,dt.
\end{align*}
In the next step we would like to obtain estimates for the marginal
densities $\mu_t$. For this purpose
we consider the  SDE
\begin{eqnarray}
\label{eq:sde-frozen}
d\widetilde X_t =  (\vphi \star \mu_t )(\widetilde X_t)\, dt + \sigma dB_t,\quad \widetilde X_t\sim \mu_0
\end{eqnarray}
which is obtained by freezing the densities \(\mu_t\) in the McKean-Vlasov equation \cref{eq:McKeanEq}.
Note that, under Assumption~\ref{assLip}, \eqref{eq:sde-frozen} is the standard SDE with uniformly bounded (by \(K_\varphi\)) and uniformly Lipschitz continuous drift function \(\vphi \star \mu_t.\) By applying \cite[Theorem~3.1]{TA20} to \eqref{eq:McKeanEq} and using our Assumption \ref{assmu0}, we deduce that

\begin{align*}
    \mu_t(x) & \le \frac{1}{\sqrt{2\pi t}} \int_\R \mu_0(y) \int_{\frac{\abs{x-y}}{\sqrt t}}^\infty z \exp\left(-\frac{(z-K_\varphi \sqrt t)^2}{2}\right) dz dy\\
    & \le \frac{\exp\left(K_\varphi^2 T/2\right)}{2\pi\sqrt{t\zeta^2}} \int_\R \exp\left(-\frac{y^2}{2\zeta^2}\right) \exp\left(-\frac{(x-y)^2}{2t}\right) dy\\
             & = \frac{\exp\left(K_\varphi^2 T/2\right)}{\sqrt{2\pi (\zeta^2 + t)}} \exp\left(-\frac{x^2}{2(\zeta^2 + t)}\right).
\end{align*}
Therefore, we get
\begin{align*}
    R(A) \le 8 K_\varphi^2 \exp\left(K_\varphi^2 T/2\right) \exp\left(-\frac{A^2}{2(\zeta^2+T)}\right).
\end{align*}
Next, note that the equation \eqref{eq:McKeanEq} can be written as
\begin{eqnarray*}
X_t =X_0+ \int_0^t (\vphi \star \mu_s )(X_s)\, ds + \sigma B_t
\end{eqnarray*}
meaning that \(\widehat{\mu}_{t}(u)=\widehat{\mu}_{0}(u)\psi_t(u)\) where \(\psi_t\) is the characteristic function of  the random variable \(\int_0^t (\vphi \star \mu_s )(X_s)\, ds + \sigma B_t.\) Thus, under Assumption~\ref{assmu0}, we conclude that
\begin{align} \label{muhatbounded}
|\widehat{\mu}_{t}(u)|\leq \exp(-u^2\zeta^2/2).
\end{align}
Furthermore, by \cite[Remark 3.3]{TA20}, \(\|\mu_t\|_{\infty}\leq 2\|\mu_0\|_{\infty}+K_{\varphi}.\) As a consequence we derive

\begin{align*}
    \inf_{g\in S_N} \norm{g-\varphi}_\star^2 & \lesssim \frac{K_{\varphi}}{A_N}\sum_{k=D_N+1}^{\infty}|c_{k}|^2 \exp\left(-\frac{k^2\zeta^2}{A_N^2}\right)\\
    & + \exp\left(K_\varphi^2 T/2\right) \exp\left(-\frac{A_N^2}{2(\zeta^2+T)}\right)
\end{align*} 
Hence under our choice of $A_N$ and $D_N$ we obtain that 
\begin{eqnarray*}
\inf_{g\in S_N} \|g-\vp\|_{\star}^2\lesssim \frac{\exp(K_\vphi^2 T/2)}{N}.
\end{eqnarray*}
This completes the proof of Theorem \ref{approxerr}. 

\subsection{Proof of Proposition \ref{prop1}}
First, we present some preliminary results.

\subsubsection{Some properties of $\Psi$}
We begin by pointing out that $\Psi$ is invertible. Indeed, this immediately follows from the representation 
$\Psi=(\lan e_i, e_j\ran_{\star})_{1\leq i,j\leq D_N}$, since $(e_j)_{1\leq j\leq D_N}$ is an orthonormal system.
The invertibility of the matrix $\Psi$ implies that we can transform $(e_j)_{1\leq j\leq D_N}$ into an orthonormal system with respect to $\lan \cdot,\cdot \ran_{\star}$.
\begin{lem}
	\label{ONBchange}
	The collection $\bar e=(\bar e_1,\dots,\bar e_{D_N})^\top $ given by 
	\begin{equation}
		\bar e=\Psi^{-1/2}e,	
		\end{equation}
	where $e=(e_1,\dots, e_{D_N})^\top$, is an orthonormal system with respect to $\lan \cdot,\cdot \ran_{\star}$.
\end{lem}
\begin{proof}
	Denoting the Kronecker delta as $\delta_{ij}$, we have
	\begin{equation}
		\sp{\bar e_i}{\bar e_j}_{\star}=\sp{(\Psi^{-1/2}e)_i}{(\Psi^{-1/2}e)_j}_{\star}=\left(\Psi^{-1/2}\Psi \Psi^{-1/2}\right)_{ij}=\delta_{ij}.
		\label{ONBstar}
	\end{equation}
\end{proof}

\subsubsection{Matrix concentration inequality for independent observations}
First of all, we recall a result from \cite[Corollary 6.1]{PMT14}, which provides a concentration inequality for operator norms of self-adjoint matrices.

\begin{theo}(\cite[Corollary 6.1]{PMT14}) \label{tropp}
    Let $(X_i)_{1\leq i\leq N}$ be a sequence of independent random variables taking values in a Polish space
    $\mathcal X$, and 
    let $H$ be a function
    from $\mathcal X^N$ into the space of $m\times m$ real self-adjoint matrices $\mathbb H^m$. Assume that   
    there exists a sequence $(A_i)_{1\leq i\leq N}\subset \mathbb H^m$ that satisfies	
        \bee
        A_i^2 \succcurlyeq \left( H(x_1,\ldots,x_{l-1},x_l,x_{l+1}\ldots,x_N) -
        H(x_1,\ldots,x_{l-1},x'_l,x_{l+1}\ldots,x_N)  \right)^2
        \eee
    for each index $1\le i\le N$, where $x_l$ and $x'_l$ range over all possible values of $X_l$.  Define 
    $\nu^2=\norm{\sum_{i=1}^N A_i^2}_{\text{\rm op}}$. Then it holds that 
        \bee
        \P\left( \norm{H(X_1,\ldots,X_N) - \E{H(X_1,\ldots,X_N)}}_{\text{\rm op}} > t \right) 
        \leq m\exp\left( - \frac{t}{2\nu^2}\right).
        \eee
\end{theo} 

\noindent
We begin by considering the random matrix $\Psiiid$ which has been constructed from the i.i.d.\ observations
$\BX^{i,N}$, $i=1,\ldots, N$, drawn from the mean field limit 
\cref{eq:McKeanEq} instead of mutually dependent particles (thus, we work under $\overline{\P}^N$). That is, 
    \bee
    (\Psiiid)_{jk} =  \frac{1}{NT} \sum_{i=1}^{N} \int_0^T e_j \star  
    \mu_t^N (\BX_t^{i,N}) e_k \star  \mu_t^N (\BX_t^{i,N}) dt
    \eee
for $1\leq j,k\leq D_N$. As a slight abuse of notation, we denote the corresponding bilinear form as $\sp\cdot\cdot_N$. In what follows, we prove that \cref{prop1} holds for $\Psiiid$ and then transition to the particle system in the next section.
\begin{proof}[Proof of \cref{prop1}] (The independent case)
	(i) Here we would like to
		prove the inequality
		\begin{align} \label{opineaq}
			\prob{\norm{\Psi^{-1/2}\Psiiid\Psi^{-1/2}-\Psi^{-1/2}\E{\Psiiid}\Psi^{-1/2}}_\opr\ge\frac 12}&\le D_N\exp\left(-\frac{c_{\eta,T}\eta NT}{4L_N^2\norm{\Psi^{-1}}_\opr^2}\right).
		\end{align}
		via \cref{tropp}.
		Recall that by \cref{ONBchange} $$\left(\Psi^{-1/2}\Psiiid\Psi^{-1/2}\right)_{jk}=\sp{\bar e_j}{\bar e_k}_{N},$$
		where $\bar e=\Psi^{-1/2} e$. We consider the space $\mathcal X = C([0,T],\R)$ and the map $H$ given by 
		\begin{align*}
			H\colon \left(C([0,T],\R)\right)^N&\to \mathbb H^{D_N},\\
			\left(x_1,\dots,x_N\right)&\mapsto\left(\frac{1}{NT}\sum_{i=1}^{N}\int_0^T \bar e_j\star\delta_{x(t)}(x_i(t))\bar e_k\star\delta_{x(t)}(x_i(t))\d t\right)_{1\le j,k\le D_N}
		\end{align*}
		where $\delta_{x(t)}=\frac1N\sum_{j=1}^N\delta_{x_j(t)}$. We have that
		\begin{align*}
		&\Psi^{-1/2}\Psiiid\Psi^{-1/2}-\Psi^{-1/2}\E_{\overline{\P}^N}\left[\Psiiid\right]\Psi^{-1/2} \\[1.5 ex]
		&= H\left(\BX^{1,N}, \ldots, \BX^{N,N}\right) -\E_{\overline{\P}^N}\left[H\left(\BX^{1,N}, \ldots, \BX^{N,N}\right)\right].
		\end{align*}
		For  $x=\left(x_1,\dots,x_N\right)\in \left(C([0,T],\R)\right)^N$ we denote by $\tilde x$ the same vector where the $l$-th entry is replaced by ${x_l}^\prime$. We need to study $\Delta H_{jk}:= \left(H(x)-H(\tilde x)\right)_{jk}$, which is given by
		\begin{align*}
		\Delta H_{jk} &= \frac{1}{TN^3} \sum_{i,r_1,r_2=1}^N \int_0^T 
		\bar e_j(x_i(t)-x_{r_1}(t))\bar e_k(x_i(t)-x_{r_2}(t)) \\
		& -\bar e_j(\tilde x_i(t)-\tilde x_{r_1}(t))\bar e_k(\tilde x_i(t)-\tilde x_{r_2}(t)) dt.
		\end{align*}
		If either none or all of the indices are equal to $l$, the difference vanishes. Therefore, there are at most $3N(N-1)$ non-zero terms. Consequently, for any $1\leq l \leq N$, we can bound the difference as
		\begin{equation*}
			|\Delta H_{jk}|\le\frac{6}{N}\norm{\bar e_j}_\infty\norm{\bar e_k}_\infty.
		\end{equation*}
		Moving on to the spectral norm of $\Delta H$, we have for any $y\in \R^{D_N}$:
		\begin{align*}
			y^\top\Delta H^2y&=\sum_{j,k=1}^{D_N}y_jy_k\sum_{l=1}^{D_N}\Delta H_{jl}\Delta H_{kl}\le \frac{36}{N^2}\sum_{j,k=1}^{D_N} |y_jy_k|\sum_{l=1}^{D_N}\norm{\bar e_j}_\infty\norm{\bar e_k}_\infty\norm{\bar e_l}_\infty^2\\	
			&\le\frac{36}{N^2}\left(\sum_{j=1}^{D_N}\norm{\bar e_j}_\infty^2\right)^2\norm{y}^2
			\le \frac{36}{N^2}L_N^2\norm{\Psi\inv}_\opr^2\norm y^2.
		\end{align*}
		This means 
		$$A_l^2:=\frac{36}{N^2}L_N^2\norm{\Psi^{-1}}_\opr^2I_{D_N}$$ 
		satisfies the condition $A_l^2\succcurlyeq \Delta H^2$ for all $1\le l\le N$. We also define 
		$\nu^2:=\frac{36}{N}L_N^2\norm{\Psi^{-1}}_\opr^2$. Now we apply the concentration inequality of \cref{tropp} and deduce the statement \cref{opineaq}, which completes this part of the proof. \\ \\
		(ii) To finish the proof of \cref{prop1}(i) in the independent case we observe that
		\begin{align*}
			\norm{\Psi^{-1/2}\Psiiid\Psi^{-1/2}
			-I_{D_N}}_\opr&\le \norm{\Psi^{-1/2}\Psiiid\Psi^{-1/2}-
			\Psi^{-1/2}\E_{\overline{\P}^N}\left[\Psiiid\right]\Psi^{-1/2}}_\opr
			\\& +\norm{\Psi^{-1/2}\E_{\overline{\P}^N}\left[\Psiiid\right]\Psi^{-1/2}-I_{D_N}}_\opr.
		\end{align*}
		Thus we only need to show the convergence
		\bee \label{conve2}
		\norm{\Psi^{-1/2}\E_{\overline{\P}^N}\left[\Psiiid\right]\Psi^{-1/2}-I_{D_N}}_\opr \to 0 \qquad \text{as } 
		N\to \infty.
		\eee
		Since the particles are exchangeable and independent, we have
		\begin{align*}
		\E_{\overline{\P}^N}\left[({\Psiiid})_{jk}\right] &= \frac{1}{TN^3} \sum_{i,r_1,r_2=1}^N \int_0^T 
		\E_{\overline{\P}^N}\left[e_j(\BX_t^{i,N}- \BX_t^{r_1,N} ) e_k(\BX_t^{i,N}- \BX_t^{r_2,N} )\right] dt \\
		& = \frac{1}{TN^2} \int_0^T \left(
		(N-1)^2\E_{\overline{\P}^N}\left[e_j\star\mu_t(\BX^{1,N}_t)e_k\star\mu_t(\BX^{1,N}_t)\right]
		+ e_j(0)e_k(0)\right. 
		\\[1.5 ex]
		&\quad + \left. (N-1)\left(\E_{\overline{\P}^N}\left[e_j\star\mu_t(\BX^{1,N}_t)\right]e_k(0)+\E_{\overline{\P}^N}
        \left[e_k\star
		\mu_t(\BX^{1,N}_t)\right]e_j(0)\right) \right) dt \\
		& = \frac{(N-1)^2}{N^2} \Psi_{jk} + \frac{e_j(0)e_k(0)}{N^2} \\[1.5 ex]
		&\quad + \frac{N-1}{TN^2} \int_0^T 
		\left(\E_{\overline{\P}^N}\left[e_j\star\mu_t(\BX^{1,N}_t)\right]e_k(0)+\E_{\overline{\P}^N}\left[e_k\star
		\mu_t(\BX^{1,N}_t)\right]e_j(0)\right)  dt.
		\end{align*} 
		Hence, we deduce that 
		\bee
		\left|\E_{\overline{\P}^N}\left[({\Psiiid})_{jk}\right]  - \Psi_{jk} \right| \lesssim \frac{1}{N} \left((I_{D_N})_{jk}
		+ \norm{e_j}_\infty\norm{e_k}_\infty \right). 
		\eee
		Now, the spectral norm can be bounded through \cref{assm}:
        \begin{equation}\label{eq: concineq_det}
            \begin{aligned}
    			\norm{I_{D_N}-\Psi^{-1/2}\E{\Psiiid}\Psi^{-1/2}}_\opr
    			&\lesssim \frac{1}{N}\left(1+\norm{\Psi\inv}_\opr\left(\sum_{j,k=1}^{D_N}\norm{e_j}_\infty^2\norm{e_k}_\infty^2\right)^{1/2}\right)\\
    			&\lesssim\frac {L_N\norm{\Psi\inv}_\opr}N\\
    			&\lesssim\frac{1}{\sqrt N\log(N)}.
		  \end{aligned}
        \end{equation}
		This gives the claim \cref{conve2} and thus the validity of 
		\cref{prop1}(i) in the independent case. \\ \\
		(iii) On $\Lambda_{N}^c$, the following inequalities hold:
		\begin{align*}
			c_{\eta,T}\frac{NT}{\log(NT)}&\le L_N^2\left(\norm{\Psi^{-1}}_\opr+\norm{\Psi^{-1}-\Psiiid^{-1}}_\opr\right)^2 \\
  & \le2L_N^2\norm{\Psi^{-1}}_\opr^2+2L_N^2\norm{\Psi^{-1}-\Psiiid^{-1}}_\opr^2.
		\end{align*}
		So, if \cref{assm} holds, we deduce that 
		\begin{equation*}
			\norm{\Psi^{-1}-\Psiiid^{-1}}_\opr^2\ge\frac{c_{\eta,T}}{4L_N^2}\frac{NT}{\log(NT)}\ge \norm{\Psi^{-1}}_\opr^2,
		\end{equation*}
		and \cite[Proposition 4 (ii)]{CG20a} yields the claim.
\end{proof}

\subsubsection{From independent observations to the particle system}
To transfer the results of \cref{prop1} from the independent case to the particle system introduced in
\cref{model}, we use an elegant transformation of measure argument investigated in \cite{DM22a}. We recall the
probability measures $\P^N$ and
$\overline{\P}^N$ that have been introduced in Section \ref{secPr2.3}. According to 
\cite[Theorem 18, eq.\ (35)]{DM22a}, there exists a $k\in\N$ such that
\begin{equation}
	\label{Hoffmann}
	\P^N(A)\le C_\varphi^{k(k+1)/8}\overline{\P}^N(A)^{k/4}\quad \forall A\in\mathcal F_T,
\end{equation}
 where $C_\varphi$ is a constant independent of $N$ and bounded by
\begin{equation}
	C_\varphi\le 1+\sum_{p=1}^{\infty}(\delta a_\varphi)^p\norm\varphi_{\text{Lip}}^{2p},
\end{equation}
where $a_\varphi$ depends exponentially on $\norm\varphi_{\text{Lip}}$ and $\delta$ needs to be chosen. Clearly, 
$$\delta<\left(2a_\varphi\norm\varphi_{\text{Lip}}^{2}\right)\inv$$ 
gives us an upper bound for $C_\varphi$, which is uniform over the class 
$\Phi$. Therefore, we have
\begin{align*}
	\P^N\left(\norm{\Psi^{-1/2}\Psi_{N}\Psi^{-1/2}-I_{D_N}}_\opr>\frac 12\right)
	&\lesssim \overline{\P}^N\left(\norm{\Psi^{-1/2}\Psiiid\Psi^{-1/2}-I_{D_N}}_\opr>
	\frac 12\right)^{k/4}\\
	&\lesssim D_N^{k/4} \exp\left(-\frac{kc_{\eta,T}}{16L_N^2\norm{\Psi\inv}_\opr^2}\eta NT\right),
\end{align*}
meaning that \cref{prop1} is true for the particle system \cref{model}. This completes the proof 
of Proposition \ref{prop1}.

\subsection{Proof of Theorem \ref{th1}}
We begin by considering the $\R^{D_N}$-valued random vector $E_N=Z_{N}-\sp{e}{\varphi}_N$, that is
$$(E_N)_j=\frac{\sigma}{NT}
\sum_{i=1}^N\int_0^T  e_j \star \mu_t^N(X_t^{i,N}) dW_t^i.$$ 
By the It\^o isometry,
\begin{equation}
	\label{EmPsi} \E\left[(E_N)_j(E_N)_k\right]=\frac{\sigma^2}{NT}\Psi_{jk} + O(N^{-2}).
\end{equation}
\begin{lem}
	\label{Em}For $E_N$ defined as above, we have
	\begin{equation*}
		\E\left[\norm{E_N}^4\right]\lesssim\frac{D_NL_N^2}{(NT)^2}.
	\end{equation*}
\end{lem}
\begin{proof}
	Using Jensen's inequality, the Burkholder-Davis-Gundy inequality, and independence of the $W^i$, we get
	\begin{align*}
		\E\left[\norm{E_N}^4\right]&\le \sigma^4 D_N\sum_{j=1}^{D_N}\E\left[\left(E_N\right)_j^4\right]\lesssim
		\sigma^4 D_N \sum_{j=1}^{D_N}\E\left[\left(\frac1{(NT)^2}\sum_{i=1}^N\int_0^T
		\left( e_j \star \mu_t^N(X_t^{i,N}) \right)^2\d t\right)^2\right]\\
		&\le \frac{\sigma^4D_N}{N^2 T^3}\sum_{j=1}^{D_N} \sum_{i=1}^N
				\int_0^T\E\left[\left(e_j \star \mu_t^N(X_t^{i,N}) \right)^4\right]\d t\\
		&\le\frac{\sigma^4 D_N}{(NT)^2}\sum_{j=1}^{D_N}\norm{e_j}_\infty^4\lesssim \frac{D_NL_N^2}{(NT)^2},
	\end{align*}
	which is the claim.
\end{proof}

For the asymptotic behaviour in $\norm\cdot_\star$ and $\norm\cdot_N$, we consider orthogonal projections and their error with respect to our introduced norms. We begin with some intermediate estimates.


\begin{prop}
	\label{projections}
	Let $\Pi^N: L^2(\|\cdot\|_\star) \to S_N$ (resp. $\Pi^{\star}$) be the orthogonal projection with respect to $\norm{\cdot}_N$ (resp. $\norm{\cdot}_{\star}$) and recall that $\Pi^{\star}\vphi =\vphi_{\star}$. Suppose that \cref{assm} is satisfied. Then we have
	\begin{enumerate}
		\item[(i)] $\E\left[\norm{\vphi_{N}-\Pi^N\varphi}_N^2\normalfont 
  1_{\Lambda_N\cap \Omega_N}\right]\lesssim \frac{D_N}{NT}.$
		\item[(ii)]	$\E\left[\norm{\vphi_{N}-\Pi^N\varphi}_N^2\normalfont 1_{\Lambda_N\cap \Omega_N^c}\right]\lesssim\frac1{NT}.$
		\item[(iii)] $\E\left[\norm{\Pi^N\left(\varphi-\vphi_{\star}\right)}_\star^2\normalfont 1_{\Omega_N}\right]\lesssim N^{-1/2}(1+L_N)\|\varphi-\vphi_{\star}\|_{\star\star}^2
  +O(N^{-1}) +o\left(\|\varphi-\vphi_{\star}\|_{\star}^2\right)$.
	\end{enumerate}
\end{prop}
\begin{rem}
	Depending on the choice of $\eta$, (ii) actually decays faster than stated. However, as (i) is the slowest term in \cref{th1}, a convergence rate of $(NT)^{-1}$ is sufficient. We will see in our proofs that this requires $\eta\ge 3$, see \cref{smallesteta}. Considering we are analysing the projection error in its respective norm, the proof of this lemma hardly differs from that of \cite[Lemma 6.3]{CG20b}. \qed
\end{rem}
\begin{proof}[Proof of \cref{projections}]
(i) As $\Pi^N\varphi-\varphi=\argmin_{f\in S_N}\norm{f-\varphi}_N$, a straightforward computation gives
		\begin{equation}
			\Pi^N\varphi=\sum_{j=1}^{D_N}(\widehat a_N)_je_j,\quad \widehat a_N=\Psi_{N}^{-1}\sp{e}{\varphi}_N.
			\label{PiN}
		\end{equation}
		By definition of $E_N$,\ $\theta_{N}-\widehat a_N=\Psi_{N}^{-1}E_N$. Hence,
		\begin{equation}
			\label{PiNchange}
			\norm{\vphi_{N} -\Pi^N\varphi}_N^2=\norm{\sum_{j=1}^{D_N}\left(\Psi_{N}^{-1}E_N\right)_je_j}_N^2
			=E_N^\top\Psi_{N}^{-1}E_N.
		\end{equation}
		On $\Omega_N$, all eigenvalues of $\Psi^{-1/2}\Psi_{N}\Psi^{-1/2}$ are in $[\frac 12, \frac 32]$, which implies that
		\begin{equation}
			\norm{\Psi^{1/2}\Psi_{N}^{-1}\Psi^{1/2}}_\opr\le 2,
			\label{PsiEV}
		\end{equation}
		such that $E_N^\top\Psi_{N}^{-1}E_N 1_{\Omega_N}\le 2E_N^\top\Psi^{-1}E_N$. Together with \cref{EmPsi} and \cref{Em}, we have deduced by \cref{assm}:
		\begin{align*}
			\E\left[\norm{\vphi_{N}-\Pi^N\varphi}_N^2 1_{\Lambda_N\cap\Omega_N}\right]&
			\le 2\E\left[E_N^\top\Psi^{-1}E_N\right]
			\lesssim\sum_{j,k=1}^{D_N}\left(\Psi^{-1}\right)_{jk}\left(\frac1{NT} \Psi_{jk}+\frac1{N^2}\right)
			\\&\lesssim \frac{D_N}{NT}+\frac{D_N\norm{\Psi\inv}_\opr}{N^2}\lesssim \frac{D_N}{NT}+\frac{D_N}{L_N\sqrt{N^3\log(NT)}}
			\\&\lesssim\frac{D_N}{NT}.
		\end{align*}
(ii) By \cref{prop1}, \cref{PiNchange} and \cref{Em}, we have
		\begin{align*}
			\E\left[\norm{\vphi_{N}-\Pi^N\varphi}_N^2 1_{\Lambda_N\cap\Omega_N^c}\right]&
			\le\E\left[\norm{\Psi_{N}^{-1}}_\opr^2 1_{\Lambda_N}\norm{E_N}^4\right]^{1/2}\prob{\Omega_N^c}^{1/2}
			\\&\le\sqrt{D_N}(NT)^{-(\eta/2+1)}.
		\end{align*}
		(iii) We define $g=\varphi-\vphi_{\star}$ and write the projection of $\Pi^Ng$ according to \cref{PiN} as an inner product. Using \cref{ONBchange}, Parseval's identity, and \cref{PsiEV} we deduce the following inequalities on $\Omega_N$.
		\begin{equation*}
			\begin{aligned}
				\norm{\Pi^Ng}_\star^2 &=\norm{
    {e^{\top}}\Psi_{N}^{-1}\sp{ e}{g}_N}^2_\star 1_{\Omega_N}=
				\norm{\bar {e}^{\top} \Psi^{1/2}\Psi_{N}^{-1}\Psi^{1/2}}\sp{\bar e}{g}_\star^2 
				\\&=\norm{\Psi^{1/2}\Psi_{N}^{-1}\Psi^{1/2}\sp{\bar e}{g}_N}^2 
				\le 4\norm{\sp{\bar e}{g}_N}^2
				\\& \le 4 \norm{\Psi^{-1}}_\opr \sum_{j=1}^{D_N} \sp{e_j}{g}_N^2.
			\end{aligned}
		\end{equation*}
  		Due to the identity $g=\varphi-\vphi_{\star}$, we have that $\sp{e_j}{g}_{\star}=0$ for all $j=1,\ldots, D_N$. Hence,  
		\begin{align} \label{zeroide}
			\sum_{j=1}^{D_N}\E\left[\sp{e_j}{g}_N^2\right] = \sum_{j=1}^{D_N}\E\left[\left(\sp{e_j}{g}_N - \sp{e_j}{g}_{\star}\right)^2\right].
		\end{align}
  For the latter term we want to use the decomposition \cref{Ustati}. However, we note that $\|g\|_{\infty}$ is not necessarily uniformly bounded in $N$, and thus we need a precise estimate of this norm. We recall that $\|\vphi\|_{\star}\leq \|\vphi\|_{\infty} \leq K_{1}$ for all $\vphi\in\Phi$. Also note the identity
  \[
  \varphi_{\star} = \sum_{j=1}^{D_N} \sp{\bar{e}_j}{\vphi}_{\star}\bar{e}_j. 
  \]
  Consequently, we deduce that 
  \begin{align*}
  \left\|\varphi_{\star}  \right\|_{\infty}& \leq  \sum_{j=1}^{D_N} |\sp{\bar{e}_j}{\vphi}_{\star}|\cdot
  \|\bar{e}_j\|_{\infty} \leq \left( \sum_{j=1}^{D_N} \sp{\bar{e}_j}{\vphi}_{\star}^2 \right)^{1/2} 
  \left( \sum_{j=1}^{D_N} \|\bar{e}_j\|_{\infty}^2 \right)^{1/2}  \\[1.5 ex]
  & \leq L_N^{1/2} \norm{\Psi^{-1}}_\opr^{1/2} \left( \sum_{j=1}^{D_N} \sp{\bar{e}_j}{\vphi}_{\star}^2 \right)^{1/2} .  
  \end{align*}
  On the other hand, we have
  \[
  \left\|\varphi_{\star} \right\|_{\star}^2=  \sum_{j=1}^{D_N} \sp{\bar{e}_j}{\vphi}_{\star}^2 
  \leq  \left\|\varphi  \right\|_{\star}^2 \leq K_{1}^2. 
  \]
  Hence, we obtain the inequality
  \begin{align} \label{infnormes}
  \left\|g \right\|_{\infty} \lesssim \left(1+ L_N^{1/2} \norm{\Psi^{-1}}_\opr^{1/2}\right). 
  \end{align}
  In the next step, we use the decomposition \cref{Ustati} and the polarisation identity $\sp{f}{g}_{\star}
  =(\|f+g\|_{\star}^2 - \|f-g\|_{\star}^2)/4$, to get  
  \begin{align} \label{decompoa}
  \sp{e_j}{g}_N - \sp{e_j}{g}_{\star} = U_{1,N}(g,e_j) + U_{2,N}(g,e_j) + U_{3,N}(g,e_j) + O\left(N^{-1}\|g\|_{\infty} \|e_j\|_{\infty}\right),
  \end{align}
  where $U_{k,N}(g,e_j):= (U_{k,N}(g+e_j) -  U_{k,N}(g-e_j))/4$ for $k=1,2,3$.  Now, we need to estimate the variance of the terms $U_{k,N}(g,e_j)$. For this purpose we will use the same change of measure device as proposed in Section \ref{secPr2.3}. In particular, due to Cauchy-Schwarz inequality \cref{changeofmeasu}, it suffices to find an appropriate bound for the fourth moment of $U_{k,N}(g,e_j)$, $k=1,2,3$, under the probability measure $\overline{\P}^N$. 

  We start with the U-statistics $U_{2,N}(g,e_j)$ and $U_{3,N}(g,e_j)$. Replacing $g$ (resp. $e_j$) by $g/\|g\|_{\infty}$ (resp. by $e_j/\|e_j\|_{\infty}$), we conclude as in Section \ref{secPr2.3}:
  \begin{align} 
  \E_{\overline{\P}^N} \left[ U_{2,N}(g,e_j)^4 \right]^{1/2} &\lesssim \|g\|_{\infty}^2 \|e_j\|_{\infty}^2 N^{-2},
  \nonumber \\[1.5 ex]
   \E_{\overline{\P}^N} \left[ U_{3,N}(g,e_j)^4 \right]^{1/2} &\lesssim \|g\|_{\infty}^2 
  \|e_j\|_{\infty}^2 N^{-3}. \label{ustatestim}
  \end{align}
  Next, we treat the term $U_{1,N}(g,e_j)$. Under $\overline{\P}^N$ we deal with i.i.d. random variables. Observe that for a statistic $W_N:=N^{-1} \sum_{i=1}^N Z_i$, where $Z_i$'s are i.i.d. random variables with $\E[Z_1]=0$ and $\E[Z_1^4]<\infty$, we have that $\E[W_N^4]= 3N^{-2}\E[Z_1^2]^2 + N^{-3}\E[Z_1^4]$. Using this observation and applying H\"older inequality, a straightforward computation shows that 
  \begin{align} \label{iidesti}
  \E_{\overline{\P}^N}\left[U_{1,N}(g,e_j)^4\right]^{1/2} \lesssim \|e_j\|_{\infty}^2 
  \left(\|g\|_{\star}^2 + \|g\|_{\star\star}^2 \right)\left( N^{-1} + N^{-3/2} \|e_j\|_{\infty}^2 \|g\|_{\infty}^2
  \right).
  \end{align}
  Now, we put everything together. According to \cref{changeofmeasu}, we deduce that 
  \[
  \E_{\P^N} \left[U_{k,N}(g,e_j)^2 \right] \leq C \E_{\overline{\P}^N}\left[U_{1,N}(g,e_j)^4\right]^{1/2}
  \qquad \text{for } k=1,2,3.
  \]
  Plugging the inequalities \cref{ustatestim} and \cref{iidesti}  into \cref{zeroide}, and taking into account the bound in \cref{infnormes}, we conclude that 
  \begin{align*}
\E\left[\norm{\Pi^N\left(\varphi-\varphi_{\star}\right)}_\star^2\normalfont 1_{\Omega_N}\right]
&\lesssim \norm{\Psi^{-1}}_\opr^2 L_N^2 \left(N^{-2} +N^{-3}\right) 
\\
&+ N^{-1} \norm{\Psi^{-1}}_\opr L_N
\left(\|g\|_{\star}^2 + \|g\|_{\star\star}^2 \right) \\
&+ N^{-3/2}\norm{\Psi^{-1}}_\opr^2 L_N^3 \left(\|g\|_{\star}^2 + \|g\|_{\star\star}^2 \right). 
  \end{align*}
  Applying Assumption \ref{assm}, we obtain the statement of the proposition. 
\end{proof}

\noindent
Now, we proceed with the proof of Theorem \ref{th1}.
Let us start with the asymptotic result in the empirical norm. We split
	\begin{equation*}
		\label{projsplit}
		\begin{aligned}
			\norm{\wphi_{N}-\varphi}_N^2&=\inf_{f\in S_N}\norm{f-\varphi}_N^2+\norm{\wphi_{N}-\Pi^N\varphi}_N^2\\&	
			=\inf_{f\in S_N}\norm{f-\varphi}_N^2+\norm{\vphi_{N}-\Pi^N\varphi}_N^2\left( 1_{\Lambda_N\cap\Omega_N}+ 1_{\Lambda_N\cap\Omega_N^c}\right)+\norm{\varphi}_N^2 1_{\Lambda_N^c}\\
			&=:\inf_{f\in S_N}\norm{f-\varphi}_N^2+ \operatorname{I}_1 + \operatorname{I}_2 + \operatorname{I}_3.
		\end{aligned}
	\end{equation*}
	Recall that $\E[\operatorname{I}_1]$ and $\E[\operatorname{I}_2]$ are treated in \cref{projections}. Furthermore, we have $\E{\operatorname{I}_3}\lesssim1/N$ due to \cref{prop1}. This gives us
	\begin{equation*}
		\E\left[\norm{\wphi_{N}-\varphi}_N^2\right]\lesssim \E\left[\inf_{f\in S_N}\norm{f-\varphi}_N^2\right]+\frac{D_N}{NT}.
	\end{equation*}
	Taking the supremum over $\varphi\in \Phi$ gives the first claim. 
 
 For the asymptotic results in $\norm{\cdot}_\star$, we split
	\begin{align*}
		\label{estsplit}
		\norm{\wphi_{N}-\varphi}_\star^2&=\norm{\wphi_{N}-\varphi}_\star^2\left( 1_{\Lambda_N\cap\Omega_N}+ 1_{\Lambda_N\cap\Omega_N^c}\right)+\norm{\varphi}_\star^2 1_{\Lambda_N^c}\\
		&=: \operatorname{II}_1 + \operatorname{II}_2 + \operatorname{II}_3.
	\end{align*}
	For $\operatorname{II}_1$, we set $g=\varphi-\varphi_{\star}$. Note that $\Pi^N\Pi^{\star}=\Pi^{\star}$, therefore
	\begin{equation*}
		\Pi^N\varphi-\varphi=\Pi^N(g+\Pi^\star\varphi)-\varphi=\Pi^Ng-g.
	\end{equation*}
	So we split the terms in $\operatorname{II}_1$ into
	\begin{align*}
		\norm{\wphi_{N}-\varphi}_\star^2&\le\norm{\wphi_{N}-\Pi^N\varphi}_\star^2+\norm{\Pi^N\varphi-\varphi}_\star^2
		\\&\le \norm{\wphi_{N}-\Pi^N\varphi}_\star^2+\norm{\Pi^Ng}_\star^2 + \norm{g}_\star^2.
	\end{align*}
	Since $\norm\cdot_N$ and $\norm\cdot_\star$ are equivalent on $\Omega_N$, we have
	\begin{equation*}
		\E\left[\norm{\wphi_{N}-\Pi^N\varphi}_\star^2 1_{\Lambda_N\cap\Omega_N}\right]\le
		2\E\left[\norm{\wphi_{N}-\Pi^N\varphi}_N^2 1_{\Lambda_N\cap\Omega_N}\right]\lesssim\frac{D_N}{NT}.
	\end{equation*}
	By \cref{projections} (iii), we have
	\begin{equation*}
		\E\left[\norm{\Pi^Ng}_\star^2 1_{\Lambda_N\cap\Omega_N}\right]+\norm{g}_\star^2
		\lesssim (1+o(1))\norm g_\star^2 + N^{-1/2}(1+L_N)\norm g_{\star\star}^2 +O(N^{-1}).
	\end{equation*}
	Moving on to $\operatorname{II}_2$, we have
	\begin{align*}
		\E\left[\norm{\wphi_{N}-\varphi}_\star^2 1_{\Lambda_N\cap\Omega_N^c}\right]\le
  \E\left[\norm{\vphi_{N}}_\star^2 1_{\Lambda_N\cap\Omega_N^c}\right]+\norm{\varphi}_\star^2\prob{\Omega_N^c}.
	\end{align*}
	Recalling the definition of $\vphi_{N}$ given in \cref{LSE} and that $(Z_{N})_j=(E_N)_j+\langle e_j,\varphi \rangle_N$, we have
	\begin{equation*}
		\begin{aligned}
			\norm{\vphi_{N}}^2_\star&=\left(\Psi_{N}^{-1} Z_{N}\right)^\top\Psi\Psi_{N}^{-1} Z_{N}\le\norm{\Psi_{N}^{-1}}_\opr^2\norm{\Psi}_\opr\norm{Z_{N}}^2
			\\&=
			\norm{\Psi_{N}^{-1}}_\opr^2\norm{\Psi}_\opr\left(\norm{\sp{ e}{\varphi}_N}^2+\norm{E_N}^2\right).
		\end{aligned}
	\end{equation*}
	Since $\varphi$ is bounded,
	\begin{align*}
		\sum_{j=1}^{D_N}\E\left[\sp{e_j}{\varphi}_N^4\right]&\le \frac1{(NT)^2}\sum_{j=1}^{D_N}\int_0^T\E\left[e_j\star\mu^N_t(X^{1,N}_t)^4\varphi\star\mu^N_t(X^{1,N}_t)^4\right]\d t
		\\&\le\sum_{j=1}^{D_N}\frac{\norm{e_j}^4_\infty}{(NT)^2}\int_0^T\E\left[\varphi\star\mu^N_t(X^{1,N}_t)^4\right]\d t
		\le\frac{L_N^2\norm\varphi_\infty^4}{(NT)^2},
	\end{align*}
	such that
	\begin{align*}
		\E\left[\left(\norm{\sp{ e}{\varphi}_N}^2+\norm{E_N}^2\right) 1_{\Lambda_N\cap\Omega_N^c}\right]&
		\le\left(\E\left[\norm{\sp{ e}{\varphi}_N}^4\right]^{1/2}+\E\left[\norm{E_N}^4\right]^{1/2}\right)\prob{\Omega_N^c}^{1/2}
		\\&\le \left(\frac{\norm\varphi_\infty^2L_N}{NT}+\frac{\sqrt D_NL_N}{NT}\right)(NT)^{-\frac{\eta-1}{2}}
		\\&\le\left(\frac{\norm\varphi_\infty^2}{\sqrt{D_N}}+1\right)L_N(NT)^{-\frac{\eta}{2}}.
	\end{align*}
	Finally, with \cref{assm} and \cref{Em}, and the fact that $\norm{\Psi}_\opr\le L_N$, we obtain 
	\begin{equation*}
		\label{smallesteta}
	\begin{aligned}
		\E\left[\norm{\vphi_{N}}_\star^2 1_{\Lambda_{N}\cap\Omega_{N}^c}\right]
		&\le\E\left[\norm{\Psi_{N}^{-1}}_\opr^2\norm{\Psi}_\opr\left(\norm{\sp{e}{\varphi}_N}^2+\norm{E_{N}}^2\right) 1_{\Lambda_{N}\cap\Omega_{N}^c}\right]
		\\&\le c_{\eta,T}\left(\frac{\norm\varphi_\infty^2}{\sqrt D_N}+1\right)(NT)^{-(\eta-1)/2}.
	\end{aligned}
	\end{equation*}
 Last but not least, we have that 
 \[
\E[\operatorname{II}_3] \lesssim \P(\Lambda_N^c)
 \]
which completes  the proof  due to Proposition \ref{prop1}.

\subsection{Proof of \cref{propPsi}}
Let us denote by $\lambda_{\min}(Q)$ the smallest eigenvalue of a symmetric matrix $Q\in \R^{D_N \times D_N}$. Since 
\begin{align*}
\lambda_{\min}(Q_1Q_2)&\geq
\lambda_{\min}(Q_1) \lambda_{\min}(Q_2)\qquad \text{and} \\[1.5 ex]
\lambda_{\min}\left(T^{-1}\int_0^T Q_t dt\right) &\geq 
T^{-1}\int_0^T \lambda_{\min}(Q_t) dt,
\end{align*} 
for all positive semi-definite matrices $Q_1,Q_2, (Q_t)_{t\geq0}$, we obtain via    \cref{Psiidenti}:
\begin{align*}
\lambda_{\min}(\Psi) &\geq \frac{1}{2A_NT} \int_0^T  \min_{k=1,\ldots,D_N} (\widehat{\mu}_t(k/A_N))^2 
\lambda_{\min}(R_t)dt \\[1.5 ex]
&\geq \frac{\exp\left(-2D_N^2/(c_2A_N^2)\right)}{2A_NT} \int_0^T  g_2^2(t) 
\lambda_{\min}(R_t)dt
\end{align*}
where $R_t \in \R^{D_N\times D_N}$ is a Toeplitz matrix given as $R_t^{kl}=\widehat{\mu}_t((l-k)/A_N)$. To deduce a lower bound for 
$\lambda_{\min}(R_t)$, we introduce the discrete-time Fourier transform 
\[
F(\omega):= \sum_{n\in \Z} \widehat{\mu}_t(n/A_N) \exp(-i\omega n), \qquad \omega \in [-\pi,\pi].
\] 
A well-known result states that 
\[
\lambda_{\min}(R_t)\geq \min_{\omega \in [-\pi,\pi]} F(\omega),
\]
and hence we need to compute $F(\omega)$. For this purpose, we will apply 
the Poisson summation formula. Define the function $f(x):=\widehat{\mu}_t(x/A_N) \exp(-i\omega x)$ for $x\in \R$ and $\mathcal{F}f(z):= \int_{\R} f(x) \exp(-i2\pi zx)dx$. The Poisson summation formula states that 
\[
F(\omega) = \sum_{n\in \Z} f(n) = \sum_{n\in \Z} \mathcal{F}f(n).
\]
By the Fourier inversion formula we obtain the identity
\[
\mathcal{F}f(n) = \int_{\R} \widehat{\mu}_t(x/A_N) 
\exp\left(-ix(\omega+2\pi n)\right) dx = \frac{A_N}{\sqrt{2\pi}} 
\mu_t\left(A_N(\omega +2\pi n)/\pi\right). 
\]
Due to Assumption \ref{lowecond}, we deduce that 
\begin{align*}
F(\omega) &\geq  \frac{A_N g_1(t)}{\sqrt{2\pi}} 
\sum_{n\in \Z} \exp\left( - \frac{A_N^2(\omega +2\pi n)^2}{c_1\pi^2}\right)\\[1.5 ex]
&\geq \frac{A_N g_1(t)}{\sqrt{2\pi}} \exp\left( - \frac{A_N^2}{c_1}\right)(1+o(1)).
\end{align*}
Putting things together, we conclude that 
\[
\lambda_{\min}(\Psi) \geq \frac{\exp\left(-2D_N^2/(c_2A_N^2)\right)
\exp\left(-A_N^2 /c_1 \right)}{2\sqrt{2\pi} T} \int_{0}^T
g_1(t) g_2(t)^2 dt. 
\]
This implies the statement of Proposition \ref{propPsi}.

\subsection{Proof of Proposition \ref{doublestar}}
Consider the decomposition 
\[
\vphi = \sum_{j=1}^{\infty} c_j e_j= \sum_{j=1}^{D_N} c_j e_j
+ \sum_{j=D_N+1}^{\infty} c_j e_j=: \vphi_1 + \vphi_2.
\]
In this scenario it holds that 
\[
\vphi - \Pi^{\star} \vphi = \vphi_2 - \Pi^{\star} \vphi_2.
\]
It trivially holds that 
\[
\|\vphi_2\|_{\star \star}^2 \leq \|\vphi_2\|_{\infty}^2 \lesssim 
A_N^{-1}\left(\sum_{j=D_N+1}^{\infty} |c_j| \right)^2.
\]
To estimate the term 
$\|\Pi^{\star} \vphi_2\|_{\star \star}^2$, we introduce the decomposition
\[
\Pi^{\star} \vphi_2 = \sum_{j=1}^{D_N} \langle \bar{e}_j, \vphi_2\rangle_{\star} \bar{e}_j.
\]
Now, we deduce that 
\[
\|\Pi^{\star} \vphi_2\|_{\star \star}^2 \lesssim
\|\Pi^{\star} \vphi_2\|_{\infty}^2.
\]
We obtain the inequality
\[
\|\Pi^{\star} \vphi_2\|_{\infty} \leq \sum_{j=1}^{D_N} |\langle \bar{e}_j, \vphi_2\rangle_{\star}|  \cdot \|\bar{e}_j\|_{\infty}
\]
In the next step we will find a bound for the term $|\langle \bar{e}_j, \vphi_2\rangle_{\star}|$. First, we note 
\[
\langle \bar{e}_j, \vphi_2\rangle_{\star} = \sum_{k=1}^{D_N} 
\sum_{l=D_N+1} \Psi^{-1/2}_{jk} c_l \sp{e_k}{e_l}_{\star}.
\]
We already know that 
\[
\sp{e_k}{e_l}_{\star}=\frac{1}{2A_NT} \int_0^T \widehat{\mu}_t(-k/A_N) \widehat{\mu}_t(l/A_N) \widehat{\mu}_t((l-k)/A_N) dt
\]
Due to Assumption \ref{assmu0}, the leading order of the above term is achieved for $l=D_n+1$ and $k=D_n/2$. Hence, we conclude that 
\[
|\langle \bar{e}_j, \vphi_2\rangle_{\star}| \lesssim D_N \|\Psi^{-1}\|^{1/2}_{\text{op}}
\exp\left(-\frac{3D_N^2\zeta^2}{4A_N^2} \right) \sum_{j=D_N+1}^{\infty} |c_j|
\]
Consequently, we obtain the inequality
\[
\|\Pi^{\star} \vphi_2\|_{\infty} \lesssim A_N^{-1/2} D_N^2 
\|\Psi^{-1}\|_{\text{op}} \exp\left(-\frac{3D_N^2\zeta^2}{4A_N^2}\right)
\sum_{j=D_N+1}^{\infty} |c_j|,
\]
which completes the proof.

\appendix

\section{Maximal inequalities for \(U\)-statistics}
\noindent
Let \(\mathcal{G}\) be a class of real-valued functions on \(\mathbb{R}^K\) and \((U_N(g):\) \(g\in \mathcal{G})\) be \(U\)-process of the form 
\begin{eqnarray*}
U_N(g)=\frac{1}{(N)_K}\sum_{i_1,\ldots,i_K=1}^N g(X_{i_1},\ldots, X_{i_K})
\end{eqnarray*}
where \(X_1,\ldots,X_N\) is a sequence of iid random variables with common distribution \(\P\)  and \((N)_K=N(N-1)\cdot\ldots\cdot (N-K+1).\) If for each  \(g\in \mathcal{G},\) 
\[
\E_{X\sim \P}[g(x_1,\ldots,x_{i-1},X,x_{i+1},\ldots,x_K)]=0,\quad i=1,\ldots,K,
\]
for arbitrary \(x_1,\ldots,x_{i-1},x_{i+1},\ldots,x_K\in \mathbb{R},\) then the process \((U_N(g),\) \(g\in \mathcal{G})\) is called degenerate. 
\begin{theo}
\label{thm:ustat-max}
Suppose that $\mathcal{G}$ is a  class of uniformly bounded functions with \(\sup_{g\in \mathcal{G}}\|g\|_\infty\leq G\) and the \(U\)-process \((U_N(g),\) \(g\in \mathcal{G})\) is degenerate. Then it holds %
\begin{align*}
 \Bigl\{\mathbb{E}\sup_{g\in\mathcal{G}}\left|N^{K/2}U_N(g)\right|^{q}\Bigr\}^{1/q} & \lesssim Gp^K\,  \mathsf{DI}(\mathcal{G},\|\cdot\|_\infty,\psi_{p/K})
\end{align*}
for any \(p\geq q.\)
\end{theo}
\begin{proof}
We can apply the maximal inequality from \cite[Section~5]{Sher94} and note that for uniformly bounded classes \(\mathcal{G}\)
\[
d_{U_N}:=\bigl[U_N\bigl(|f-g|^2\bigr)\bigr]^{1/2}\leq \|f-g\|_{\infty} 
\]
 as well as \(U_N\bigl(|g|^2\bigr)\leq G^2\) with probability \(1\) for all \(f,g\in \mathcal{G}.\)
\end{proof}

\subsection*{Acknowledgement} 
Mark Podolskij and Shi-Yuan Zhou gratefully acknowledge financial support of ERC Consolidator Grant 815703 “STAMFORD: Statistical Methods for High Dimensional Diffusions”.

\bibliographystyle{chicago}

\end{document}